\documentclass[11pt,reqno, a4paper]{amsart}
\usepackage{graphicx} % Enable draft mode for graphicx
\usepackage{mathrsfs}
\usepackage{color}
\usepackage{epsfig,subfigure}
\usepackage[outdir=./]{epstopdf}
\usepackage{enumerate}
\usepackage{amssymb,amsmath,amsthm,graphicx,amsfonts}
\usepackage{hyperref}
\usepackage[margin=1.0in]{geometry}
\usepackage[ruled,linesnumbered]{algorithm2e}

\usepackage{titlesec} 
\titleformat{\section}{\vskip10pt\large\bfseries}{\thesection.}{0.5em}{\centering\vspace{5pt}}
\titleformat{\subsection}{\vskip10pt\normalsize\bfseries}{\thesubsection.}{0.5em}{}
\newtheorem{theorem}{Theorem}[section]

\newtheorem{lemma}[theorem]{Lemma}

\theoremstyle{definition}

\newtheorem{example}[theorem]{Example}

\def\d{\mathrm{d}}
\allowdisplaybreaks

\numberwithin{equation}{section}

\begin{document}

\title[]{Computing rough solutions of the stochastic\\ 
nonlinear wave equation}

\author[]{Jiachuan Cao\,\,}

\author[]{Buyang Li\,\,}
\address{\hspace*{-12pt}Jiachuan Cao and Buyang Li: 
Department of Applied Mathematics, The Hong Kong Polytechnic University,
Hong Kong. 
{\rm Email address: {\tt jiachuan.cao@polyu.edu.hk} and {\tt buyang.li@polyu.edu.hk}}}

\author[]{Katharina Schratz}
\address{\hspace*{-12pt}Katharina Schratz:
Laboratoire Jacques-Louis Lions, Sorbonne Université,
Bureau : 16-26-315, 4 place Jussieu, Paris 5ème. \newline
{\it E-mail address}: {\tt katharina.schratz@sorbonne-universite.fr}}


%\date{\today}


\maketitle
\vspace{-20pt}

\begin{abstract}\noindent
{\small 
The regularity of solutions to the stochastic nonlinear wave equation plays a critical role in the accuracy and efficiency of numerical algorithms. Rough or discontinuous initial conditions pose significant challenges, often leading to a loss of accuracy and reduced computational efficiency in existing methods. In this study, we address these challenges by developing a novel and efficient numerical algorithm specifically designed for computing rough solutions of the stochastic nonlinear wave equation, while significantly relaxing the regularity requirements on the initial data. By leveraging the intrinsic structure of the stochastic nonlinear wave equation and employing advanced tools from harmonic analysis, we construct a time discretization method that achieves robust convergence for initial values \((u^{0}, v^{0}) \in H^{\gamma} \times H^{\gamma-1}\) for all \(\gamma > 0\). Notably, our method attains an improved error rate of \(O(\tau^{2\gamma-})\) in one and two dimensions for \(\gamma \in (0, \frac{1}{2}]\), and \(O(\tau^{\max(\gamma, 2\gamma - \frac{1}{2}-)})\) in three dimensions for \(\gamma \in (0, \frac{3}{4}]\), where \(\tau\) denotes the time step size. These convergence rates surpass those of existing numerical methods under the same regularity conditions, underscoring the advantage of our approach. To validate the performance of our method, we present extensive numerical experiments that demonstrate its superior accuracy and computational efficiency compared to state-of-the-art methods. These results highlight the potential of our approach to enable accurate and efficient simulations of stochastic wave phenomena even in the presence of challenging initial conditions.
}
\end{abstract}

%\tableofcontents

\setlength\abovedisplayskip{4.5pt}
\setlength\belowdisplayskip{4.5pt}

\section{The stochastic wave equation}
We consider the numerical solution of the following stochastic nonlinear wave equation in a rectangular domain $\mathcal{O}=[0,1]^d$ with the periodic boundary condition, with a multiplicative noise of It\^{o} type, i.e., 
\begin{equation}\label{eq:1-1}
\left\{
\begin{array}{lll}
 \partial_{tt}u-\Delta u=f(u)+\sigma(u)\d W\quad &\text{in }\mathcal{O}\times (0,T], \\[2mm]
 u|_{t=0}=u^{0}\quad \text{and}\quad \partial_{t} u|_{t=0}=v^{0} \quad &\text{in }\mathcal{O} , 
\end{array}
\right.
\end{equation}
where $ W(t) $ denotes an $ \mathbb{R} $-valued Wiener process defined on the filtered probability space $(\Omega,\mathcal{F},\{\mathcal{F}_t\}_{0\leq t\leq T},\mathcal{P})$; the nonlinear functions $f(u)$ and $\sigma(u)$ are assumed to be smooth with respect to $u$. 

The nonlinear stochastic wave equation emerges in various fields of physics, including the motion of a strand of DNA suspended in a liquid, heat flow around a ring, and the internal dynamics of the sun \cite{Dalang1962,Thomas2012}. Since an analytical expression for the solutions to the stochastic nonlinear wave equation in \eqref{eq:1-1} is not available, there has been significant interest among researchers in developing numerical approximations to the solution. A key distinction from deterministic problems lies in the term $\sigma(u)\d W$, which represents the increment of a Wiener process scaled by $\sqrt{\d t}$. This scaling introduces considerable challenges in constructing efficient numerical methods for solving the stochastic wave equation. Over the past few decades, various time-stepping schemes, such as the Euler scheme, Crank-Nicolson scheme, Störmer-Verlet leapfrog scheme, trigonometric schemes, and others, have been extensively studied in the literature \cite{ACLW2016,Banjai2021,CHSS2020,Cohen2013,Cohen2016,Cox2024,Feng2024,Hong2022,Kovacs2013,Quer2006,Walsh2006,Wang2014,Wang2015}. 

%The nonlinear stochastic wave equation arises from various areas of physics, such as the motion of a strand of DNA floating in a liquid, heat flow around a ring, and the internal structure of the sun \cite{Dalang1962,Thomas2012}. Since the analytical expression of solutions to the stochastic nonlinear wave equation in \eqref{eq:1-1} is not available, the study of numerical approximations to the solution has attracted significant interest from researchers. Compared to deterministic problems, the main difference is the term $\d W(t)$, which is the difference of a Wiener process in time, scaled as $\sqrt{\d t}$. This scaling introduces significant difficulties in constructing efficient numerical approaches for solving stochastic equations. Over the past few decades, time-stepping schemes based on the Euler scheme, Crank-Nicolson scheme, St\"{o}rmer-Verlet leapfrog scheme, trigonometric schemes, and others have been extensively studied in the literature \cite{ACLW2016,Banjai2021,CHSS2020,Cohen2013,Cohen2016,Cox2024,Feng2024,Hong2022,Kovacs2013,Quer2006,Walsh2006,Wang2014,Wang2015}. 

For instance, Walsh \cite{Walsh2006} was the first to prove strong convergence of the St\"ormer--Verlet leapfrog scheme for the one-dimensional stochastic wave equation with multiplicative noise, establishing an error bound of $O(\tau^{\frac{1}{2}})$, where $\tau$ denotes the stepsize of time discretization. Strong convergence with an error bound of $O(\tau)$ was achieved in \cite{ACLW2016,Banjai2021,Wang2015} through the use of the stochastic trigonometric scheme and a variant of the St\"ormer--Verlet leapfrog scheme. Recently, Feng, Panda, and Prohl developed a $1.5$-order numerical scheme in \cite{Feng2024} for the stochastic wave equation, inspired by second-order time-stepping methods used for deterministic wave equations. 

In general, existing methods for solving the nonlinear stochastic wave equation with multiplicative noise require the solution pair \((u, \partial_t u)\) to be at least in \(C([0,T];H^1 \times L^2)\). This regularity is required for the solution to achieve \(\frac{1}{2}\)-H\"older continuity in time, which plays a crucial role in ensuring the convergence rates of the corresponding numerical methods; see, for instance, Lemma 3.3 in \cite{Wang2015}, Lemma 4.6 in \cite{Banjai2021}, or Lemma 3.3 in \cite{Feng2024}. However, in practice, the nonlinear stochastic wave equation is well-posed for initial data in \(H^{\gamma} \times H^{\gamma-1}\) whenever $\gamma\ge 0$ (see, for example, \cite[Theorem 2.1]{Wang2015} for the well-posedness result of the problem with \( L^2 \times H^{-1} \) initial data).  In particular, rough initial data—potentially with discontinuities—can prevent the solution of \eqref{eq:1-1} from attaining the regularity required by the existing numerical methods. This lack of regularity can significantly reduce the convergence order of the algorithms and, consequently, diminish computational efficiency. The numerical computation of rough solutions, especially those below the regularity threshold of \(H^{1/2} \times H^{-1/2}\), remains a significant challenge in the context of the nonlinear stochastic wave equation. Addressing this issue is crucial for advancing numerical methods that can handle solutions with low regularity or discontinuities.

To address the issue of order reduction when computing rough solutions, various low-regularity algorithms have been developed and analyzed in recent years for deterministic nonlinear dispersive equations. Examples include approaches for the Korteweg-de Vries (KdV) equation \cite{Hofmanova-Schratz-2017,Li-Wu-Yao-2021,Wu-Zhao-IMA,Wu-Zhao-BIT} and the nonlinear Schrödinger equation \cite{BS,OS18,ORS20,RS21}. For the nonlinear wave equation, the low-regularity algorithms constructed in \cite{RS21,LSZ} have second-order convergence in $H^{1} \times L^{2}$ under the regularity condition $(u^{0}, v^{0}) \in H^{1+\frac{d}{4}} \times H^{\frac{d}{4}}$ for spatial dimensions $d = 1, 2, 3$. 
%They employed the transformation 
%$$
%w = u - i(-\Delta)^{-\frac{1}{2}}\partial_t u,
%$$
%which converts the nonlinear wave equation into a first-order formulation. This technique enabled the design of a low-regularity integrator that achieves second-order convergence in the energy norm $H^{1} \times L^{2}$ under the regularity condition $(u^{0}, v^{0}) \in H^{\frac{7}{4}} \times H^{\frac{3}{4}}$ in three dimensions. Additionally, a different low-regularity integrator was developed in \cite{LSZ}, which leverages a novel cancellation structure. This integrator also achieves second-order convergence in $H^{1} \times L^{2}$ under the regularity condition $(u^{0}, v^{0}) \in H^{1+\frac{d}{4}} \times H^{\frac{d}{4}}$ for spatial dimensions $d = 1, 2, 3$.
For the nonlinear wave equation with cubic nonlinearity in one dimension, a symmetric low-regularity integrator was constructed in \cite{CLLY2024}, achieving second-order convergence in \(H^{\gamma} \times H^{\gamma-1}\) under the condition \((u^{0}, v^{0}) \in H^{\gamma} \times H^{\gamma-1}\) for \(\gamma > \frac{1}{2}\). In the case of Lipschitz continuous nonlinearities in one and two dimensions, a high-frequency recovery low-regularity integrator was proposed in \cite{CLLY2024}, achieving convergence rates of \(2\gamma\) and \(1.5\gamma\) in one and two dimensions, respectively, for rough solutions of the nonlinear wave equation under the regularity condition \((u^{0}, v^{0}) \in H^{\gamma} \times H^{\gamma-1}\) with \(\gamma \in (0, 1/2]\).

In addition to these works, Bronsard, Bruned, and Schratz introduced a new class of numerical methods for approximating dispersive equations with randomness in low-regularity settings in \cite{ABS2024}. These methods employ resonance-based discretization, as described in \cite{BS}, to approximate the expectation of the solution when randomness is present in the initial data. Building on this, Armstrong-Goodall and Bruned developed a class of resonance-based schemes in \cite{Armstrong2023}, specifically designed to achieve targeted convergence orders for computing rough solutions to the stochastic Schr\"odinger equation (half-order convergence for $H^1$ initial data) and the Manakov system (half-order convergence for $H^3$ initial data). As far as we know, the numerical computation of rough solutions to the stochastic nonlinear wave equation with a multiplicative noise, especially those below the regularity threshold of \(H^{1/2} \times H^{-1/2}\), remains a significant challenge. 

The objective of this work is to develop a novel numerical scheme for the stochastic nonlinear wave equation to achieve higher-order convergence while significantly relaxing the regularity requirements for the initial data. A key aspect that underpins the feasibility of our objective is the ability to construct an algorithm without relying on the H\"older continuity of the exact solution in the time direction, extending the idea of low-regularity integrators as well as low- and high-frequency decomposition techniques to the stochastic nonlinear wave equation. Unlike previous research on numerical schemes for solving \eqref{eq:1-1}, we do not approximate $(u(t), \partial_t u(t))$ over the time interval $[t_n, t_n + \tau]$ using $(u(t_n), \partial_t u(t_n))$. Instead, we employ the solution operator of the linear wave equation, denoted as $e^{sL}$, to construct the approximation of $(u(t), \partial_t u(t))$ as $e^{(t - t_n)L}(u(t_n), \partial_t u(t_n))$, where $L$ is a linear operator defined in \eqref{system_notation}. The advantage of this approach lies in the fact that it allows us to provide a more accurate error estimate, which can be expressed as follows: 
\begin{align}\label{eq:new_approximation}
	\mathbb{E}\left\|\begin{pmatrix}
		u(t)\\[2mm] \partial_{t}u(t)
	\end{pmatrix}-e^{(t-t_n)L}\begin{pmatrix}
	u(t_n)\\[2mm] \partial_{t}u(t_n)
	\end{pmatrix}\right\|_{H^1\times L^2}^{2} \lesssim |t-t_n| ,
\end{align}
which is valid for low-regularity solutions $(u,\partial_t u)\in C([0,T];H^{\gamma}\times H^{\gamma-1})$ with $\gamma$ possibly approaching zero. By utilizing the estimate in \eqref{eq:new_approximation}, we utilize a second-order Taylor expansion of the nonlinear term $\Sigma(u(t), \partial_t u(t)) := (0, \sigma(u(t)))^\top$ at $e^{(t-t_n)L}(u(t_n), \partial_t u(t_n))$. By leveraging the structure of the linear operator $e^{(t-t_n)L}$, we derive the following approximation result:
\begin{align}
	\mathbb{E} \left\|\Sigma\begin{pmatrix}
		u(t)\\[2mm] \partial_{t}u(t)
	\end{pmatrix}-\Sigma\left(e^{(t-t_n)L}\begin{pmatrix}
		u(t_n)\\[2mm] \partial_{t}u(t_n)
	\end{pmatrix}\right)\right\|_{H^1\times L^2}^{2} \lesssim |t-t_n|^2,
\end{align}
which holds for $(u, \partial_t u) \in C([0,T]; H^{\gamma} \times H^{\gamma - 1})$ with $\gamma$ possibly approaching zero. This significantly relaxes the regularity requirements necessary for the numerical solution to achieve a certain convergence rate. 

Furthermore, by employing a low- and high-frequency decomposition of the solution, combined with a detailed error analysis that leverages the structure of the stochastic nonlinear wave equation, we construct a first-order exponential integrator that achieves an error of \(O(\tau^{2\gamma -})\) in one and two dimensions for \((u^0, v^0) \in H^{\gamma} \times H^{\gamma - 1}\) with \(\gamma \in (0, \frac{1}{2}]\). This convergence rate is twice as fast as that of existing methods under the same regularity condition. In three dimensions, the proposed method is shown to converge with an error of \(O(\tau^{\max(\gamma, 2\gamma - \frac{1}{2} -)})\) for \((u^0, v^0) \in H^{\gamma} \times H^{\gamma - 1}\) with \(\gamma \in (0, \frac{3}{4}]\). Importantly, in both cases, this work provides the first proof of convergence rates for computing rough solutions of the stochastic wave equation in the low-regularity regime below \(H^{\frac{1}{2}} \times H^{-\frac{1}{2}}\).

%The proposed method can be implemented with the Fourier spectral method in the spatial discretization. However, direct spatial discretization with the Fourier spectral method requires using a larger number of Fourier basis functions in order to mitigate the spatial discretization error from affecting the accuracy of the time discretization. This would consequently increase the overall computational cost. To further improve the accuracy of spatial discretization without essentially increasing the computational cost, we draw inspiration from low-regularity algorithms for deterministic nonlinear wave equations discussed in \cite{CLLY2024}, and accordingly propose a high-frequency recovery algorithm which computes the high-frequency and low-frequency parts of the numerical solution separately, as shown in \eqref{eq:1st_HRLRI}. 
%%This leads to an error bound of $O(N^{-(1+1/d)\gamma})$ for the fully discrete numerical solution under the CFL condition $\tau=O(N^{-1})$, where $N$ denotes the number of Fourier basis to be used in each dimension. %Numerical tests illustrate the efficiency and significant advantages of the proposed method. 

%It is important to note that the primary focus of this article is to address temporal discretization and achieve high-order convergence while reducing the regularity requirements imposed by rough initial data. 
To make our approach clear to the readers, we consider an $\mathbb{R}$-valued Wiener process in \eqref{eq:1-1} and provide an error estimate for the proposed semi-discretization in time. The theoretical results can, in fact, be extended to more general Wiener processes, such as the Q-Wiener process, under suitable regularity assumptions. Furthermore, the arguments presented in this article can be adapted to a full discretization in space using the Fourier spectral method, by incorporating the error estimates for deterministic nonlinear wave equations discussed in \cite{CLLY2024}. These extensions are left for future work and are not discussed in detail in the current paper.

The rest of the paper is organized as follows: In Section \ref{section:main-results}, we introduce the basic notations and main results of this paper, including the proposed algorithm and its error estimates. Section \ref{section:preliminary} covers preliminary technical results that will be essential for constructing and analyzing our method. In Section \ref{sec:1st_order}, we construct our numerical scheme and prove its convergence properties. Finally, in Section \ref{section:numerical_test}, we outline the implementation of the fully discrete scheme and present numerical examples that support the theoretical results while demonstrating the effectiveness of the proposed method.
%In Section \ref{sec:1st_order}, we construct and discuss the first-order scheme. The higher-order scheme and its filtered version are constructed and analyzed in Section \ref{section:higher_order} and Section \ref{section:filtered_higher_order}, respectively. Finally, we provide the algorithms implementing the fully discrete scheme of this paper in Section \ref{section:numerical_test} and give some numerical examples based on the proposed method to support the theoretical results proved in this paper and to illustrate the effectiveness of the proposed method.

\section{Notation and Main results}\label{section:main-results}
%\subsection{The numerical scheme}
%\subsection{Notations and definitions}

\subsection{Basic settings}

By introducing $U=(u,\partial_tu)^\top$, we reformulate \eqref{eq:1-1} into the following first-order system of equations:
\begin{equation}\label{system}
\left\{
\begin{array}{lll}
 \partial_{t}U-LU=F(U) + \Sigma(U)\d W\quad &\text{in }\mathcal{O}\times (0,T], \\[2mm]
 U|_{t=0}=U^{0}\quad &\text{in}\,\,\, \mathcal{O} , 
\end{array}
\right.
\end{equation} 
where
\begin{equation}\label{system_notation}
U=\begin{pmatrix}
u\\ \partial_{t}u
\end{pmatrix},\quad
U^{0}=\begin{pmatrix}
u^{0}\\ v^{0}
\end{pmatrix},\quad
F(U)=\begin{pmatrix}
0\\ f(u)
\end{pmatrix},\quad
\Sigma(U)=\begin{pmatrix}
0\\ \sigma(u)
\end{pmatrix},\quad L=\begin{pmatrix}
0& 1\\
\Delta & 0
\end{pmatrix}.
\end{equation}

For the simplicity of error analysis, which inspires us to construct the numerical scheme in \eqref{eq:filtered_1st_scheme}, we assume that the nonlinear functions $f$ and $\sigma$ satisfy the following conditions for some constant $C>0)$:
\begin{equation}\label{sigma_Lipschitz}
|f^{\prime}(s)|+|f^{\prime\prime}(s)|+|f^{\prime\prime\prime}(s)|+|\sigma^{\prime}(s)|+|\sigma^{\prime\prime}(s)|+|\sigma^{\prime\prime\prime}(s)|\lesssim C, \quad\text{for} \quad s\in\mathbb{R}.
\end{equation}
In order to measure the error of the numerical solution, we define the following norm and inner product of $H^{\gamma}(\mathcal{O})\times H^{\gamma-1}(\mathcal{O})$ for $\gamma\in \mathbb{R}$:
% and $\gamma\neq 1$:
\begin{equation}\label{non_energy_norm}
\|U\|_{\gamma}=\Big(\|u\|^{2}_{H^{\gamma}(\mathcal{O})}+\|v\|^{2}_{H^{\gamma-1}(\mathcal{O})}\Big)^{\frac{1}{2}}\quad \text{for}\quad U=(u,v)^\top \in H^{\gamma}(\mathcal{O})\times H^{\gamma-1}(\mathcal{O}),
\end{equation}
and
\begin{equation}\label{non_energy_inner_product}
\langle U_1, U_2\rangle_{\gamma} =\int_{\mathcal{O}}\mathcal{J}^{\gamma} u_1(x) \mathcal{J}^{\gamma} u_2(x)+ \mathcal{J}^{\gamma-1} v_1(x) \mathcal{J}^{\gamma-1} v_2(x) dx ,
\end{equation}
where $\mathcal{J}^{\gamma}:=(1-\Delta)^{\frac{\gamma}{2}}$ and $\mathcal{J}^{\gamma-1}:=(1-\Delta)^{\frac{\gamma-1}{2}}$. 

The $C_{0}$-semigroup generated by $L$ is denoted by $e^{tL}$, which is the solution operator of the linear wave equation (i.e., the map from $U^0$ to $U(t)$ in the case $F(U)\equiv\Sigma(U)\equiv 0$), satisfying the following estimate: 
\begin{equation}\label{semigroup}
\begin{array}{l}
%\|e^{tL}U\|_{1}=\|U\|_{1}\quad \text{for all}\quad  U\in H^{1}(\mathcal{O})\times L^{2}(\mathcal{O}),\\[2mm]
\|e^{tL}U\|_{r}\le C\|U\|_{r}\quad \text{for all}\quad  U\in H^{r}(\mathcal{O})\times H^{r-1}(\mathcal{O}) . 
\end{array}
\end{equation}

Under condition \eqref{sigma_Lipschitz}, it is known that for any initial state $U^0\in H^{\gamma}(\mathcal{O})\times H^{\gamma-1}(\mathcal{O})$ with $\gamma>0$, equation \eqref{system} has a unique solution $U\in L^{\infty}(0,T; H^{\gamma}(\mathcal{O})\times H^{\gamma-1}(\mathcal{O}))$ satisfying the following variation-of-constant formula (see e.g. \cite[Section 7.1]{Prato1996} and \cite[Theorem 2.1]{Wang2015}) for all $t,s\geq 0$:
\begin{equation}\label{mild_form}
U(t+s)=e^{sL}U(t)+\int_{0}^{s}e^{(s-\delta)L}F\big(U(t+\delta)\big)\d\delta+\int_{0}^{s}e^{(s-\delta)L}\Sigma\big(U(t+\delta)\big)\d W(t+\delta) . 
\end{equation}
Moreover, for $p\geq 2$, the solution has continuous sample paths  $U(t)=(u(t),v(t))^\top$ satisfying the following estimate: 
\begin{equation}\label{Lp_bound}
\mathbb{E}\left[\sup_{0\leq t\leq T}\Big(\|u(t)\|_{H^{\gamma}}^{p}+\|v(t)\|_{H^{\gamma-1}}^{p}\Big)\right]\leq K(p),
\end{equation}
where $K(p)$ is some constant depending only on $p$.

Throughout this article, we denote by $A\lesssim B$ or $A\gtrsim B$ the statement ``$A\leq CB$ for some constant $C>0$ which is independent of $\tau$ and $n$''.

\subsection{The numerical scheme and its convergence}
Since any function in the Sobolev space $L^2(\mathcal{O})$ can be expanded into a Fourier series. Accordingly, we introduce the finite-dimensional subspace 
$$
S_N=\bigg\{\sum_{n_1, \cdots, n_d=-N}^N c_{n_1, \cdots, n_d} \exp \left(2n_1 \pi x_1 i\right) \cdots \exp \left(2n_d \pi x_d i\right): c_{n_1, \cdots, n_d} \in \mathbb{C}\bigg\} ,
$$
and approximate functions in $H^s(\mathcal{O})$ by using the finite-dimensional subspace $S_N$. We denote by $\Pi_N$ the $L^2$ projection operator onto $S_N$ defined by 
\begin{align}\label{eq:def_projection}
(w-\Pi_Nw, v) = 0, \quad \forall v\in S_N, \quad w\in H^s(\Omega) , 
\end{align}
and denote $ \Pi_{>N}:= I - \Pi_{N}$ and $\Pi_{(N_1,N_2]}:=\Pi_{N_2}-\Pi_{N_1}$ for $N_2>N_1$. We denote by $I_N: H^s\to S_N$ the trigonometric interpolation  such that  
for any function $w\in H^s$, $s>\frac{d}2$,
$\left(I_N w\right)(x)=w(x)$ for $x \in D^d$, with
$$
D=\left\{\frac{n}{2 N}: n=0, \cdots, 2N-1\right\}. 
$$

Let $t_{n}=n\tau$, $n\in \{0,1,\cdots, M\}$, be a partition of the time interval $[0,T]$ with stepsize $\tau=T/M$ for some positive integer $M$. 
We propose the following filtered low-regularity exponential integrator for the stochastic wave equation: 
\begin{align}\label{eq:filtered_1st_scheme}
U^{n+1}=e^{\tau L}U^{n}+\tau e^{\tau L}\Pi_{\tau^{-1}}F(\Pi_{\tau^{-1}}U^{n})+e^{\tau L}\Pi_{\tau^{-1}}\Sigma(\Pi_{\tau^{-1}}U^{n})\Delta_{n}W, 
\end{align}
with $\Delta_nW= W(t_{n+1}) - W(t_n)$. 
This method is constructed through the error analysis in Section \ref{subsec:construction_1st_order}. 
%
%We define the energy norm and corresponding inner product in $H^{1}(\mathcal{O})\times L^{2}(\mathcal{O})$ as follows: 
%\begin{equation}\label{energy_norm}
%\|U\|_{1}=\left\|\begin{pmatrix}
%u\\v
%\end{pmatrix}\right\|_{1}=\Big(m\|u\|^{2}_{L^{2}(\mathcal{O})}+\|\nabla u\|^{2}_{L^{2}(\mathcal{O})}+\|v\|^{2}_{L^{2}(\mathcal{O})}\Big)^{\frac{1}{2}},
%\end{equation}
%and 
%\begin{equation}\label{inner_product}
%\langle U_1, U_2\rangle_{1} =\int_{\mathcal{O}}m\cdot u_1(x)u_2(x)+\nabla u_1(x) \cdot \nabla u_2(x)+ v_1(x) v_2(x) dx
%\end{equation}
%
%\subsection{Error estimates}

The main theoretical result regarding the convergence of the proposed numerical scheme is presented in the following theorem. 

\begin{theorem}\label{thm:0_norm_error_rough_data}
Let $U$ be the unique solution of \eqref{mild_form} with initial value $U^0\in H^{\gamma}(\mathcal{O})\times H^{\gamma-1}(\mathcal{O})$ under condition \eqref{sigma_Lipschitz}, and let $\{U^{n}\}_{n\geq 0}$ be the numerical solution given by \eqref{eq:filtered_1st_scheme}. Then the following error estimates hold: 
\begin{align}
%&\max_{0\leq n\leq T/\tau}\Big(\mathbb{E}\Big[\big\|U(t_{n})-U^{n}\big\|_{0}^{2}\Big]\Big)^{\frac{1}{2}}\lesssim \tau^{2\gamma-}
%\quad \mbox{for}\,\,\, \gamma\in(0,\frac{1}{2}] 
%\,\,\,\mbox{and}\,\,\, d=1,2 ,\\
%%%%%%%%%%%%%%
&\max_{0\leq n\leq T/\tau}\Big(\mathbb{E}\big[\|U(t_{n})-U^{n}\|_{0}^{2}\big]\Big)^{\frac{1}{2}}\lesssim\left\{
\begin{aligned}
&\tau^{2\gamma-}
&& \mbox{for} &&\gamma\in \Big(0,\frac{1}{2}\Big] 
&&\mbox{and}\,\,\, d=1,2 ,\\
&\tau^{\gamma}
&& \text{for}&& \gamma\in \Big(0,\frac{1}{2}\Big]
&&\mbox{and}\,\,\, d=3,\\
&\tau^{2\gamma-\frac{1}{2}-}
&& \text{for}&& \gamma\in \Big(\frac{1}{2},\frac{3}{4}\Big] 
&&\mbox{and}\,\,\, d=3.
\end{aligned}
\right.
\end{align}
\end{theorem}

The proof of Theorem \ref{thm:0_norm_error_rough_data} is presented in the following sections. Since the main difficulty in constructing the numerical algorithm arises from the multiplicative noise term $\Sigma(u)\partial_{t}W$, we will focus on the case $f(u)\equiv 0$ in the rest of this paper in order to make our idea clearer and to simplify the presentation. For the general case $f(u)\neq 0$, low-regular integrators can be constructed and analyzed in the similar way. 

\section{Preliminary estimates for the error analysis}\label{section:preliminary}

In this section, we present several technical lemmas that will be used in the subsequent section for the construction and analysis of the numerical scheme.

\begin{lemma}\label{lemma_1}
	Let $ U $ be the unique solution of \eqref{system} under the condition \eqref{sigma_Lipschitz}, with initial data $ U^{0} \in H^{\gamma}(\mathcal{O}) \times H^{\gamma - 1}(\mathcal{O}) $ for $ \gamma \geq 0 $. Then, for any $ t_{n} $ and $ 0 \leq s \leq \tau $, we have the following estimates:
	\begin{equation}\label{1st_order_error}
		\mathbb{E}\left[ \| U(t_{n} + s) - e^{sL} U(t_{n}) \|_1^2 \right] \leq C \tau,
	\end{equation}
	and
	\begin{equation}\label{2nd_order_error}
		\mathbb{E}\left[ \| U(t_{n} + s) - e^{sL} U(t_{n}) \|_1^4 \right] \leq C \tau^2,
	\end{equation}
	where $ C $ is a constant that depends only on $ \sigma $, $ K(2) $, and $ K(4) $.
\end{lemma}

\begin{proof}
Using the equality \eqref{mild_form}, the property \eqref{semigroup}, and It\^{o}'s isometry, we obtain the following:

\begin{align}\label{eq:2-2}
	\mathbb{E}\left[\|U(t_{n}+s) - e^{sL} U(t_{n})\|_{1}^{2}\right] &\leq \mathbb{E}\left[\int_{0}^{s} \| e^{(s - \delta)L} \Sigma\left(U(t + \delta)\right) \|_{1}^{2} \, d\delta\right]\notag \\
	&= \mathbb{E}\left[\int_{0}^{s} \| \Sigma\left(U(t + \delta)\right) \|_{1}^{2} \, d\delta\right] \\
	&= \mathbb{E}\left[\int_{0}^{s} \| \sigma\left(u(t + \delta)\right) \|_{L^{2}}^{2} \, d\delta\right].\notag
\end{align}

Using the condition \eqref{sigma_Lipschitz}, we get the following estimate:

\begin{align}\label{eq:2-3}
	\mathbb{E}\left[\int_{0}^{s} \| \sigma\left(u(t + \delta)\right) \|_{L^{2}}^{2} \, d\delta\right] &\lesssim \mathbb{E}\left[\int_{0}^{s} \| \sigma(0) \|_{L^{2}}^{2} + \| u(t_{n} + \delta) \|_{L^{2}}^{2} \, d\delta \right] \notag\\
	&\lesssim \tau \left( 1 + \mathbb{E}\left[ \sup_{0 \leq \delta \leq s} \| U(t_{n} + \delta) \|_{\gamma}^{2} \right] \right),
\end{align}
for each \(\gamma \geq 0\). Combining \eqref{eq:2-2} and \eqref{eq:2-3}, we obtain the estimate \eqref{1st_order_error}.

For the proof of \eqref{2nd_order_error}, recalling the variation-of-constant formula, we apply the Burkholder-Davis-Gundy (B-D-G) inequality (see, e.g., \cite{Karatzas}, Section 3.3, Theorem 3.28) to obtain

\begin{align}\label{eq:2-5}
	\mathbb{E}\left[\|U(t_{n}+s) - e^{sL} U(t_{n})\|_{1}^{4}\right] &= \mathbb{E}\left[\left\| \int_0^s e^{(s-\delta)L} \Sigma(U(t_n) + \delta) \, dW(t_n + \delta) \right\|_{1}^{4}\right] \notag \\
	&\lesssim \mathbb{E}\left[\left( \int_{0}^{s} \| e^{(s-\delta)L} \Sigma\left(U(t+\delta)\right) \|_{1}^{2} \, d\delta \right)^{2}\right] \notag \\
	&\lesssim \tau^{2} \cdot \left( 1 + \mathbb{E}\left[ \sup_{0 \leq \delta \leq s} \| U(t_{n} + \delta) \|_{\gamma}^{4} \right] \right),
\end{align}
for each $\gamma \geq 0$, which yields \eqref{2nd_order_error}.
\end{proof}

To construct low-regularity integrators for \eqref{system}, it is crucial to utilize the following calculations for the time derivatives of the auxiliary functions, as established in \cite{LSZ}.

\begin{lemma}\label{lem:cancellation_structure}
	Let $ U = (u,v)^\top $ and $ \tilde{U}(s) := e^{sL} U = (\tilde{u}(s), \tilde{v}(s))^\top $. Then the following identity holds:
	\begin{equation}\label{eq:dt_G}
		\frac{\d}{\d s} e^{-sL} \Sigma(\tilde{U}(s)) = e^{-sL} \begin{pmatrix}
			-\sigma(\tilde{u}(s)) \\[3mm]
			\sigma'(\tilde{u}(s)) \tilde{v}(s)
		\end{pmatrix}.
	\end{equation}
\end{lemma}

We now present some useful results for estimating the consistency errors.

\begin{lemma}[{Bernstein's inequality; cf. \cite[Theorem 2.2 and pp. 22]{Guo1998}}]\label{bound}
	Let $ f $ be a function such that $ \mathcal{J}^\gamma f := (1 - \Delta)^{\frac{\gamma}{2}} f \in L^p(\mathcal{O}) $ for some $ \gamma \geq 0 $ and $ 1 < p < \infty $. Then the following estimates hold:
	\begin{align*}
		\| \Pi_{\leq N} \mathcal{J}^\gamma f \|_{L^p} &\lesssim N^\gamma \| f \|_{L^p}, \\
		\| \Pi_{> N} f \|_{L^p} &\lesssim N^{-\gamma} \| \mathcal{J}^\gamma f \|_{L^p}.
	\end{align*}
\end{lemma}

%The next lemma shows the error estimate of the interpolation $I_N$, whose proof can be found in \cite{Shreve1974}.
%\begin{lemma}[{Standard error of trigonometric interpolation; cf. \cite[Theorem 11.8]{Kress1989}}]\label{lemInterp}
%Let $f$ be a function such that $f\in H^\gamma( \mathcal{O})$.
%For $0\leq s\leq \gamma$ and $ \gamma> \frac{d}{2}$, we have 
%\begin{align*}
%\| f - I_N f\|_{H^{s}} \lesssim N^{- (\gamma-s)} \| f\|_{H^\gamma}.
%\end{align*}

%\end{lemma}

%\begin{lemma}[ Error of trigonometric interpolation in the $L^{p}$ norm; see \cite{arXiv}]\label{lemInterp_2} Let $d=1$ and $f\in W^{1,p}(\mathcal{O})$ for $1<p<\infty$ then
%\begin{align*}
%\| f - I_N f\|_{L^{p}} \lesssim N^{- 1} \| f\|_{W^{1,p}}.
%\end{align*}
%\end{lemma}

%\begin{lemma}[ Negative-norm estimates for the product of two functions when $d=1,2$; see \cite{arXiv}]\label{Lemma:negative-norm}
%For $d = 1,2$, the following estimates hold: 
%\begin{align}
%\label{dualargu2}
%\| fg\|_{H^{-1}}&\lesssim \| g\|_{H^{-1}} \left( \|f\|_{L^\infty} +\| f\|_{H^{1+}}\right),\\
%\label{embedding3}
%\| fg\|_{H^{-1}}&\lesssim \| f\|_{L^{2+}}^ \gamma \| f\|_{H^{1+}}^ {1- \gamma} \| g\|_{H^{ \gamma -1}},\\
%\label{embedding2}
%\| fg\|_{H^{-1}}&\lesssim \| f\|_{L^2} \| g\|_{L^{2+}}.
%\end{align}
%\end{lemma}

\begin{lemma}[Negative-norm estimates for the product of two functions when $ d = 3 $]\label{lemma_neg_norm_3d}
	For $ d = 3 $, the following estimates hold: 
	\begin{align}
		\label{dualargu2_3d}
		\| fg \|_{H^{-1}} &\lesssim \| g \|_{H^{-1}} \left( \| f \|_{L^\infty} + \| f \|_{H^{\frac{3}{2}}} \right), \\
		\label{embedding3_3d}
		\| fg \|_{H^{-1}} &\lesssim \| f \|_{H^{\frac{1}{2}}}^{\gamma} \| f \|_{H^{\frac{3}{2}+}}^{1 - \gamma} \| g \|_{H^{\gamma - 1}}.
	\end{align}
\end{lemma}

\begin{proof}
	By the dual argument and Sobolev embedding in three dimensions (i.e., $ H^1 \hookrightarrow L^6 $, and $ H^{\frac{1}{2}} \hookrightarrow L^3 $), we have
	\begin{align*}
		\| fg \|_{H^{-1}} &= \sup_{\| w \|_{H^1} = 1} \langle g, fw \rangle \lesssim \sup_{\| w \|_{H^1} = 1} \| g \|_{H^{-1}} \| fw \|_{H^1} \\
		&\lesssim \sup_{\| w \|_{H^1} = 1} \| g \|_{H^{-1}} \left( \| f w \|_{L^2} + \| f \nabla w \|_{L^2} + \| w \nabla f \|_{L^2} \right) \\
		&\lesssim \sup_{\| w \|_{H^1} = 1} \| g \|_{H^{-1}} \left( \| f \|_{L^\infty} \| w \|_{H^1} + \| w \|_{L^6} \| \nabla f \|_{L^3} \right) \\
		&\lesssim \sup_{\| w \|_{H^1} = 1} \| g \|_{H^{-1}} \left( \| f \|_{L^\infty} + \| f \|_{H^{\frac{3}{2}}} \right) \| w \|_{H^1}.
	\end{align*}
	
	Similarly, for \eqref{embedding3_3d}, we have
	\begin{align}\label{proof0}
		\| fg \|_{H^{-1}} &= \sup_{\| w \|_{H^1} = 1} \langle g, fw \rangle 
		\lesssim \sup_{\| w \|_{H^1} = 1} \| g \|_{H^{\gamma - 1}} \| fw \|_{H^{1 - \gamma}}.
	\end{align}
	
	By the interpolation inequality and embeddings $ H^{\frac{3}{2}+} \hookrightarrow L^\infty $, $ H^1 \hookrightarrow L^6 $, and $ H^{\frac{1}{2}} \hookrightarrow L^3 $, we obtain
	\begin{align}\label{proof2}
		\| fw \|_{H^{1 - \gamma}} &\lesssim \| fw \|_{L^2}^{\gamma} \| fw \|_{H^1}^{1 - \gamma} \\
		&\lesssim \| f \|_{L^3}^{\gamma} \| w \|_{L^6}^{\gamma} \left[ \| \mathcal{J} f \|_{L^3}^{1 - \gamma} \| w \|_{L^6}^{1 - \gamma} + \| f \|_{L^\infty}^{1 - \gamma} \| \mathcal{J} w \|_{L^2}^{1 - \gamma} \right] \\
		&\lesssim \| f \|_{H^{\frac{1}{2}}}^{\gamma} \| f \|_{H^{\frac{3}{2}+}}^{1 - \gamma} \| w \|_{H^1}.
	\end{align}
	
	Substituting \eqref{proof2} into \eqref{proof0} leads to \eqref{embedding3_3d}.
\end{proof}

\begin{lemma}[Negative-norm estimates for the composition functions when $d=1,2$; see \cite{CLLY2024}]\label{lem:H_estimate}
	Let $d = 1,2$, and $(u,v)^\top \in H^{\gamma}(\mathcal{O}) \times H^{\gamma-1}(\mathcal{O})$. Assume that $\sigma$ satisfies the conditions in \eqref{sigma_Lipschitz}. Then, the following estimates hold:
	\begin{align}\label{gamma_leq_frac12}
		\|\sigma^{\prime}(\Pi_N u)\Pi_N v\|_{H^{-1}} &\lesssim N^{1-2\gamma+}, 
		&& \text{for } \gamma \in (0,\tfrac{1}{2}], \\ 
		\label{gamma_greater_frac12}
		\|\sigma^{\prime}(\Pi_N u)\Pi_N v\|_{H^{-1}} &\lesssim 1, 
		&& \text{for } \gamma \in (\tfrac{1}{2},1].
	\end{align}
\end{lemma}

\begin{lemma}[Negative-norm estimates for the composition functions when $d=3$]\label{lem:H_estimate_3d}
	Let $d = 3$, and $(u,v)^\top \in H^{\gamma}(\mathcal{O}) \times H^{\gamma-1}(\mathcal{O})$. Assume that $\sigma$ satisfies the conditions in \eqref{sigma_Lipschitz}. Then, the following estimates hold:
	\begin{align}\label{gamma_leq_frac12_3d}
		\|\sigma^{\prime}(\Pi_N u)\Pi_N v\|_{H^{-1}} &\lesssim N^{\frac{3}{2}-2\gamma+}, 
		&& \text{for } \gamma \in (0,\tfrac{3}{4}], \\ 
		\label{gamma_greater_frac12_3d}
		\|\sigma^{\prime}(\Pi_N u)\Pi_N v\|_{H^{-1}} &\lesssim 1, 
		&& \text{for } \gamma \in (\tfrac{3}{4},1].
	\end{align}
\end{lemma}

\begin{proof}
As in the proof of Lemma~\ref{lem:H_estimate} in \cite{CLLY2024}, we assume that $\log_2 N$ is an integer (otherwise, we replace $\log_2 N$ by the smallest integer larger than it). Using the triangle inequality, we have 
\begin{align}\label{star}
	\|\sigma^\prime(\Pi_N u) \cdot \Pi_N v\|_{H^{-1}}
	&\leq \sum_{k=1}^{\log_2 N - 1} 
	\|(\sigma^\prime(\Pi_{2^{k+1}} u) 
	- \sigma^\prime(\Pi_{2^k} u)) 
	\cdot \Pi_N v\|_{H^{-1}}.
\end{align}

The right-hand side of \eqref{star} can be estimated using \eqref{embedding3_3d}:
\begin{align*}
	&\quad \|(\sigma^\prime(\Pi_{2^{k+1}} u) 
	- \sigma^\prime(\Pi_{2^k} u)) 
	\cdot \Pi_N v\|_{H^{-1}} \\
	&\lesssim \|\sigma^\prime(\Pi_{2^{k+1}} u) 
	- \sigma^\prime(\Pi_{2^k} u)\|_{H^{\frac{1}{2}}}^\gamma 
	\|\sigma^\prime(\Pi_{2^{k+1}} u) 
	- \sigma^\prime(\Pi_{2^k} u)\|_{H^{\frac{3}{2}+}}^{1-\gamma} 
	\|\Pi_N v\|_{H^{\gamma - 1}} \\
	&\lesssim \|\sigma^\prime(\Pi_{2^{k+1}} u) 
	- \sigma^\prime(\Pi_{2^k} u)\|_{H^{\frac{1}{2}}}^\gamma 
	\|\sigma^\prime(\Pi_{2^{k+1}} u) 
	- \sigma^\prime(\Pi_{2^k} u)\|_{H^{1}}^{\frac{1-\gamma}{2}} 
	\|\sigma^\prime(\Pi_{2^{k+1}} u) 
	- \sigma^\prime(\Pi_{2^k} u)\|_{H^{2}}^{\frac{1-\gamma}{2}+} \\
	&\lesssim \big((2^k)^{\frac{1}{2} - \gamma}\big)^\gamma 
	\big((2^k)^{1 - \gamma}\big)^{\frac{1-\gamma}{2}} 
	\big((2^k)^{2 - \gamma}\big)^{\frac{1-\gamma}{2}+} \\
	&\lesssim (2^k)^{\frac{3}{2} - 2\gamma +}.
\end{align*}

For $\gamma \in (0, \frac{3}{4}]$, summing up the above estimate for $k = 1, \dots, \log_2 N - 1$ gives:
\begin{align}\label{proof_summation}
	\|\Pi_N(\sigma^\prime(\Pi_N u) \cdot \Pi_N v)\|_{H^{-1}} 
	&\lesssim \sum_{k=1}^{\log_2 N - 1} 
	(2^k)^{\frac{3}{2} - 2\gamma +} 
	\lesssim N^{\frac{3}{2} - 2\gamma +} \cdot \log_2 N.
\end{align}
This proves the first result of Lemma~\ref{lem:H_estimate_3d}.

For $\gamma \in (\frac{3}{4}, 1]$, note that $\frac{3}{2} - 2\gamma + < 0$, so the sum in \eqref{proof_summation} converges, and we obtain \eqref{gamma_greater_frac12_3d}, where the constant depends only on $\gamma$.

\hfill\hfill\end{proof}

\begin{lemma}[Negative-norm estimates for frequency localization]\label{lem:freq_local_err}
	Under the regularity condition \eqref{Lp_bound} for $U(t) = (u(t), v(t))^\top$, the $L^2(\mathcal{O}) \times H^{-1}(\mathcal{O})$ error between $\Sigma(U(t))$ and $\Sigma(\Pi_N U(t))$ for any $t \in [0,T]$ satisfies the following estimate:
	\begin{align}\label{eq:freq_local_err}
		\mathbb{E}\left[\left\|\Sigma(U(t))-\Sigma(\Pi_N U(t))\right\|^2_0\right] 
		\lesssim
		\begin{cases}
			N^{-4\gamma+}, & \text{for } d = 1, 2, \\
			N^{\min(-2\gamma, 1 - 4\gamma+)}, & \text{for } d = 3.
		\end{cases}
	\end{align}
\end{lemma}

\begin{proof}
	The estimate \eqref{eq:freq_local_err} for $d = 1, 2$ follows directly from Lemma~6.1 in \cite{CLLY2024}. 
	
	For $d = 3$, since 
	\[
	\left\|\Sigma(U(t)) - \Sigma(\Pi_N U(t))\right\|_0 
	\leq \|\sigma(u(t)) - \sigma(\Pi_N u(t))\|_{L^2},
	\]
	we have
	\begin{align}\label{eq4eq_1}
		\mathbb{E}\left[\left\|\Sigma(U(t))-\Sigma(\Pi_N U(t))\right\|^2_0\right] 
		&\leq \mathbb{E}\left[\|\sigma(u(t)) - \sigma(\Pi_N u(t))\|^2_{L^2}\right] \notag \\
		&\lesssim \mathbb{E}\left[\|\Pi_{>N} u(t)\|^2_{L^2}\right] \notag \\
		&\leq N^{-2\gamma} \mathbb{E}\left[\sup_{0 \leq t \leq T} \|u(t)\|^2_{H^{\gamma}}\right].
	\end{align}
	
	Next, decompose $\Sigma(U(t)) - \Sigma(\Pi_N U(t))$ dyadically as follows (with $M = \lceil \log_2 N \rceil$):
	\begin{align*}
		\Sigma(U(t)) - \Sigma(\Pi_N U(t)) 
		&= \Sigma(U(t)) - \Sigma(\Pi_{2^M N} U(t)) 
		+ \sum_{j=0}^{M-1} \Big(\Sigma(\Pi_{2^{j+1}N} U(t)) - \Sigma(\Pi_{2^j N} U(t))\Big) \\
		&=: \mathcal{G}_0 + \sum_{j=0}^{M-1} G_j.
	\end{align*}
	
	By the mean-value theorem, we obtain
	\begin{align*}
		\|G_j\|_0 
		&\lesssim \|\Sigma(\Pi_{2^{j+1}N} U(t)) - \Sigma(\Pi_{2^j N} U(t))\|_0 \\
		&\lesssim \max_{\theta} \|\sigma^\prime(\xi) \cdot (\Pi_{2^{j+1}N} - \Pi_{2^j N}) u(t)\|_{H^{-1}},
	\end{align*}
	where $\xi = \theta \Pi_{2^{j+1}N} u(t) + (1 - \theta) \Pi_{2^j N} u(t)$. Using \eqref{dualargu2_3d}, we estimate
	\begin{align*}
		\|G_j\|_0 
		&\lesssim \|(\Pi_{2^{j+1}N} - \Pi_{2^j N}) u(t)\|_{H^{-1}} 
		\left(\|\sigma^\prime(\xi)\|_{H^{\frac{3}{2}+}} + \|\sigma^\prime\|_{L^\infty}\right) \\
		&\lesssim (2^j N)^{-1-\gamma} \|u(t)\|_{H^\gamma} 
		\left((2^j N)^{\frac{3}{2} - \gamma+} \|u(t)\|_{H^\gamma} + 1\right).
	\end{align*}
	
	For $\mathcal{G}_0$, we directly estimate:
	\[
	\|\mathcal{G}_0\|_0 \lesssim \|\Pi_{>N^2} u(t)\|_{L^2} \lesssim N^{-2\gamma} \|u(t)\|_{H^\gamma}.
	\]
	
	Summing these terms, we have:
	\begin{align}\label{eq4eq}
		\mathbb{E}\left[\left\|\Sigma(U(t)) - \Sigma(\Pi_N U(t))\right\|^2_0\right] 
		&\leq M \left(\mathbb{E}\left[\|\mathcal{G}_0\|_0^2\right] + \sum_{j=0}^{M-1} \mathbb{E}\left[\|G_j\|_0^2\right]\right) \notag \\
		&\lesssim M \left(N^{-4\gamma} + \sum_{j=0}^{M-1} (2^j N)^{1 - 4\gamma+}\right) \mathbb{E}\left[1 + \|u(t)\|^4_{H^\gamma}\right] \notag \\
		&\lesssim N^{1 - 4\gamma+} \mathbb{E}\left[1 + \sup_{0 \leq t \leq T} \|u(t)\|^4_{H^\gamma}\right],
	\end{align}
	when $\gamma > \frac{1}{4}$. Here, $M$ is absorbed into $N^{0+}$ in the last inequality. 
	
	Combining \eqref{eq4eq_1} and \eqref{eq4eq} completes the proof of \eqref{eq:freq_local_err} for $d = 3$.
\end{proof}

\section{Proof of the main results}\label{sec:1st_order}

In this section, we present the construction of the numerical scheme through analyzing the error in approximating the variation-of-constant formula \eqref{mild_form}. Then we prove the convergence of the numerical scheme. 

\subsection{Construction of the numerical scheme in (\ref{eq:filtered_1st_scheme})}\label{subsec:construction_1st_order}
Writing $\Sigma(U(t_{n}+s))$ in terms of Taylor expansion at $e^{s L}U(t_{n})$, we have
\begin{equation}\label{eq:taylor_sigma}
\begin{array}{l}
\Sigma\big(U(t_{n}+s)\big)=\Sigma\big(e^{s L}U(t_{n})\big)+\Sigma^{\prime}\big(e^{s L}U(t_{n})\big)\big(U(t_{n}+s)-e^{sL}U(t_{n})\big)\\[3mm]
\qquad\qquad\qquad\quad+\Big(R_{\Sigma}(s)\big(U(t_{n}+s)-e^{sL}U(t_{n})\big)\Big)\cdot \big(U(t_{n}+s)-e^{sL}U(t_{n})\big),
\end{array}
\end{equation}
where
\begin{equation}\label{eq:sigma_prime}
\Sigma^{\prime}(V)=
\begin{pmatrix}
0& 0\\
\sigma^{\prime}(V_{1})& 0
\end{pmatrix}
\quad \text{for}\quad
V=\begin{pmatrix}
V_1\\ V_2
\end{pmatrix},
\end{equation}
and
\begin{equation}\label{eq:r_sigma}
{\displaystyle R_{\Sigma}(s)=\int_{0}^{1}\int_{0}^{1}\theta \Sigma^{\prime\prime}\Big[(1-\xi)e^{sL}U(t_{n})+\xi(1-\theta)e^{sL}U(t_{n})+\theta U(t_{n}+s)\Big]\d \xi \d\theta,}
\end{equation}
here $\Sigma^{\prime\prime}(V)=\big(\Sigma^{\prime\prime}(V)_{ijk}\big)_{2\times 2\times 2}=\big(\partial_{V_{k}}\partial_{V_{j}}\Sigma_{i}(V))_{2\times 2\times 2}$ is a third-order tensor with $\Sigma^{\prime\prime}_{211}(V)=\sigma^{\prime\prime}(V_{1})$ and $\Sigma^{\prime\prime}_{ijk}(V)=0$ for $(i,j,k)\neq (2,1,1)$. Then, substituting \eqref{eq:taylor_sigma} into \eqref{mild_form}, we obtain
\begin{equation}\label{eq:3-1}
\begin{array}{l}
{\displaystyle U(t_{n}+\tau)=e^{\tau L}U(t_{n})+\int_{0}^{\tau}e^{(\tau-s)L}\Sigma\big(e^{sL}U(t_{n})\big)\d W(t_{n}+s)}\\[3mm]
{\displaystyle \qquad\qquad\quad\;\;+\int_{0}^{\tau}e^{(\tau-s)L}\Sigma^{\prime}\big(e^{sL}U(t_{n})\big)\big(U(t_{n}+s)-e^{sL}U(t_{n})\big)\d W(t_{n}+s)+R_{1}(t_{n})}\\[3mm]
{\displaystyle \qquad\qquad\;\; :=e^{\tau L}U(t_{n})+I_{1}(t_{n})+I_{2}(t_{n})+R_{1}(t_{n}),}
\end{array}
\end{equation}
where
\begin{equation}\label{eq:r_1}
R_{1}(t_{n})=\int_{0}^{\tau}e^{(\tau-s)L}\left[\Big(R_{\Sigma}(s)\big(U(t_{n}+s)-e^{sL}U(t_{n})\big)\Big)\cdot \big(U(t_{n}+s)-e^{sL}U(t_{n})\big)\right]\d W(t_{n}+s).
\end{equation}
The remainder $R_{1}(t_{n})$ is estimated in the following lemma.

\begin{lemma}\label{lem:r_1}
Under the conditions of Theorem \ref{thm:0_norm_error_rough_data}, the following estimates hold for $d=1,2,3$: 
\begin{equation}\label{eq:estimate_r1}
\mathbb{E}\left[\|R_{1}(t_{n})\|_{1}^{2}\right]\lesssim \tau^{3}.
\end{equation}
\end{lemma}
\begin{proof}

From the expression of $\Sigma^{\prime\prime}(V)$ in \eqref{eq:sigma_prime} for each $V=(1-\xi)e^{sL}U(t_{n})+\xi(1-\theta)e^{sL}U(t_{n})+\theta U(t_{n}+s)$, we note that
\begin{equation}
\begin{array}{l}
\Big(\Sigma^{\prime\prime}(V)\big(U(t_{n}+s)-e^{sL}U(t_{n})\big)\Big)\cdot \big(U(t_{n}+s)-e^{sL}U(t_{n})\big)\\[3mm]
\,\, =\begin{pmatrix}
0\\
\sigma''(V_{1})\big(U(t_{n}+s)-e^{sL}U(t_{n})\big)^{2}_{1}
\end{pmatrix},
\end{array}
\end{equation}
where $\big(U(t_{n}+s)-e^{sL}U(t_{n})\big)_{1}$ denotes the first component of $U(t_{n}+s)-e^{sL}U(t_{n})$. Then recalling the condition \eqref{sigma_Lipschitz}, we have
% and the Sobolev embedding $L^{1+}(\mathcal{O})\hookrightarrow H^{-1}(\mathcal{O})$ for $d=1,2$
\begin{equation}
\begin{array}{l}
\Big\|\Big(R_{\Sigma}(s)\big(U(t_{n}+s)-e^{sL}U(t_{n})\big)\Big)\cdot \big(U(t_{n}+s)-e^{sL}U(t_{n})\big)\Big\|_{1}\\[2mm]
\quad \lesssim \big\|\sigma''(V_{1}) \big(U(t_{n}+s)-e^{sL}U(t_{n})\big)^{2}_{1}\big\|_{L^{2}}
\lesssim \big\|\big(U(t_{n}+s)-e^{sL}U(t_{n})\big)_{1}\big\|^2_{L^{4}}.
\end{array}
\end{equation}
Since, according to the Sobolev embedding $H^{1}(\mathcal{O})\hookrightarrow L^{4}(\mathcal{O})$ for $d=1,2,3$, 
\begin{equation}
\big\|\big(U(t_{n}+s)-e^{sL}U(t_{n})\big)_{1}\big\|^{2}_{L^{4}}\lesssim \big\|U(t_{n}+s)-e^{sL}U(t_{n})\big\|_{1}^{2},
\end{equation}
it follows that 
\begin{equation}\label{eq:3-2}
\Big\|\Big(R_{\Sigma}(s)\big(U(t_{n}+s)-e^{sL}U(t_{n})\big)\Big)\cdot \big(U(t_{n}+s)-e^{sL}U(t_{n})\big)\Big\|_{1}\lesssim \big\|U(t_{n}+s)-e^{sL}U(t_{n})\big\|_{1}^{2}.
\end{equation}
Thus, by using the It\^{o}'s isometry and estimate \eqref{2nd_order_error} in Lemma~\ref{lemma_1}, we have
\begin{equation}\label{eq:3-2-R1}
\begin{array}{l}
{\displaystyle \mathbb{E}\left[\|R_{1}\|_{1}^{2}\right]\leq\mathbb{E}\left[\int_{0}^{\tau}\Big\|\Big(R_{\Sigma}(s)\big(U(t_{n}+s)-e^{sL}U(t_{n})\big)\Big)\cdot \big(U(t_{n}+s)-e^{sL}U(t_{n})\big)\Big\|^{2}_{1}\d s\right]}\\[4mm]
{\displaystyle \qquad\qquad\;\;\lesssim \tau \sup_{0\leq s\leq \tau} \mathbb{E}\left[\big\|U(t_{n}+s)-e^{sL}U(t_{n})\big\|_{1}^{4}\right]\lesssim \tau^3.}
\end{array}
\end{equation} 
This completes the proof of Lemma \ref{lem:r_1}.
\end{proof}

Next, we show that $I_{2}(t_{n})$ is also a small term that can be dropped in the design of the numerical scheme. 
By substituting \eqref{mild_form} into the expression of $I_{2}(t_{n})$ in \eqref{eq:3-1}, we have
\begin{align}\label{eq:rewrite_I2}
	I_{2}(t_{n})=&\int_{0}^{\tau}e^{(\tau-s)L}\Big[\Sigma^{\prime}\big(e^{sL}U(t_{n})\big)\big(U(t_{n}+s)-e^{sL}U(t_{n})\big)\Big]\d W(t_{n}+s)\notag\\
	=&\int_{0}^{\tau}e^{(\tau-s)L}\left[\Sigma^{\prime}\big(e^{sL}U(t_{n})\big)\int_{0}^{s}e^{(s-\delta)L}\Sigma\big(U(t_{n}+\delta)\big)\d W(t_{n}+\delta)\right]\d W(t_{n}+s)\notag\\
	=&\int_{0}^{\tau}e^{(\tau-s)L}\left[\Sigma^{\prime}\big(e^{sL}U(t_{n})\big)\int_{0}^{s}\Big(e^{(s-\delta)L}-1\Big)\Sigma\big(U(t_{n}+\delta)\big)\d W(t_{n}+\delta)\right]\d W(t_{n}+s),
\end{align}
where we have used the property $\Sigma^{\prime}\big(e^{sL}U(t_{n})\big)\Sigma\big(U(t_{n}+\delta)\big)\equiv 0$ in the last equality of \eqref{eq:rewrite_I2}. Based on the expression in \eqref{eq:rewrite_I2}, we estimate $I_2(t_n)$ in the following lemma.

\begin{lemma}\label{lem:1_norm_I2}
Under the conditions of Theorem \ref{thm:0_norm_error_rough_data}, the following estimates hold for $d=1,2,3$: 
	\begin{equation}\label{eq:1_norm_estimate_I2}
		\mathbb{E}\left[\|I_2(t_{n})\|_{1}^{2}\right]\lesssim \tau^{4}.
	\end{equation}
\end{lemma}

\begin{proof}
By using relation \eqref{eq:rewrite_I2} and It\^{o}'s isometry, we have
	\begin{align*}
		\mathbb{E}\left[\|I_2(t_n)\|^2_1\right]\lesssim & \int_{0}^{\tau}\int_{0}^{s}\mathbb{E}\big[\big\|\Sigma^{\prime}\big(e^{sL}U(t_{n})\big)\big(e^{(s-\delta)L}-1\big)\Sigma\big(U(t_{n}+\delta)\big)\big\|_{1}^{2}\big]\d\delta \d s\\
		\lesssim &\int_{0}^{\tau}\int_{0}^{s}(s-\delta)^2 \sup_{s,\delta}\left(\mathbb{E}\left[\left\|\Sigma^{\prime}\big(e^{sL}U(t_{n})\big)L\Sigma\big(U(t_{n}+\delta)\big)\right\|_{1}^{2}\right]\right)\d\delta \d s.
	\end{align*} 
We note that
	\begin{align*}
		\Sigma^{\prime}\big(e^{sL}U(t_{n})\big)L\Sigma\big(U(t_{n}+\delta)\big)&=\begin{pmatrix}
			0 & 0\\
			\sigma^{\prime}\big(\tilde{u}(t_{n}+s)\big) & 0
		\end{pmatrix}
		\begin{pmatrix}
			0 & 1\\
			\Delta & 0
		\end{pmatrix}
		\begin{pmatrix}
			0\\
			\sigma\big(u(t_n+\delta)\big)
		\end{pmatrix}\\
		&=\begin{pmatrix}
			0\\
			\sigma^{\prime}\big(\tilde{u}(t_{n}+s)\big)\sigma\big(u(t_{n}+\delta)\big)
		\end{pmatrix},
	\end{align*}
	Thus $I_2$ can be estimated by
	\begin{align*}
		\mathbb{E}\left[\|I_2(t_n)\|^2_1\right]\lesssim & \int_{0}^{\tau}\int_{0}^{s}(s-\delta)^2 \mathbb{E}\left[\sup_{s,\delta}\left\|\sigma^{\prime}\big(\tilde{u}(t_{n}+s)\big)\sigma\big(u(t_{n}+\delta)\big)\right\|_{L^2}^{2}\right]\d\delta \d s\\[2mm]
		\lesssim & \tau^4 \sup_{\delta}\left(\mathbb{E}\left[1+\left\|U(t_{n}+\delta)\right\|_{0}^{2}\right]\right).
	\end{align*}
This proves the result of Lemma \ref{lem:1_norm_I2}. 
\end{proof}

%\begin{remark}
%	\upshape
%	From the discussions in Lemma \ref{lem:r_1} and Lemma \ref{lem:1_norm_I2}, we can also infer that 
%	\begin{align*}
%		\mathbb{E}\left[\|\Sigma(U(t_n+s))-\Sigma(e^{sL}U(t_n))\|_{1}^{2}\right]^{\frac{1}{2}}\lesssim \tau,
%	\end{align*}
%	for all $0\leq s\leq \tau$. This is one of the reasons why we can achieve high-order convergence in time.
%\end{remark}

In view of relation \eqref{eq:3-1} and Lemmas \ref{lem:r_1}--\ref{lem:1_norm_I2}, it remains to find a good approximation of $I_{1}(t_{n})$. To ensure the stability of numerical approximations, we introduce the frequency localization operator $\Pi_{\tau^{-1}}$ to decompose $I_{1}(t_{n})$ as follows: 
\begin{align}\label{eq:i_1}
I_1(t_n)&=\int_{0}^{\tau}e^{\tau L}\left(e^{-sL}\Sigma\big(e^{sL}\Pi_{\tau^{-1}} U(t_{n})\big)\right)\d W(t_{n}+s)+R_{21}(t_n)\notag\\
&=\int_{0}^{\tau}e^{\tau L}\Sigma(\Pi_{\tau^{-1}} U(t_n))\d W(t_{n}+s)+R_{21}(t_n)+R_{22}(t_n)\notag\\[2mm]
&=\Pi_{\tau^{-1}}\Sigma(\Pi_{\tau^{-1}}U(t_n))\Delta_n W +R_{21}(t_n)+R_{22}(t_n)+R_{23}(t_n),
\end{align} 
where 
\begin{equation}\label{delta_w}
\Delta_{n}W:=W(t_{n}+\tau)-W(t_{n}),
\end{equation}
and
\begin{align}
R_{21}&=\int_{0}^{\tau}e^{(\tau-s) L}\left[\Sigma\big(e^{sL}U(t_{n})\big)-\Sigma\big(e^{sL}\Pi_{\tau^{-1}} U(t_{n})\big)\right]\d W(t_{n}+s)\label{eq:def_tilde_R_21}\\
R_{22}&=\int_{0}^{\tau}e^{\tau L}\left[e^{-sL}\Sigma\big(e^{sL}\Pi_{\tau^{-1}} U(t_{n})\big)-\Sigma(\Pi_{\tau^{-1}} U(t_n))\right]\d W(t_{n}+s)\label{eq:def_tilde_R_22}\\[2mm]
R_{23}&=\left[\Sigma(\Pi_{\tau^{-1}} U(t_n))-\Pi_{\tau^{-1}}\Sigma(\Pi_{\tau^{-1}}U(t_n))\right]\Delta_n W.\label{eq:def_tilde_R_23}
\end{align}

The estimate of $R_2(t_n) = R_{21}(t_n)+R_{22}(t_n)+R_{23}(t_n)$ is presented in the following lemma. 
\begin{lemma}\label{lem:r_2}
Under the conditions of Theorem \ref{thm:0_norm_error_rough_data}, the following estimates hold for $d=1,2,3$: 
\begin{align}\label{eq:0_norm_R2}
\mathbb{E}\left[\|R_2(t_n)\|_0^2\right]\lesssim 
\left\{
\begin{aligned} 
&\displaystyle \tau^{1+4\gamma-} &&\text{for}\quad \gamma\in \Big(0,\frac{1}{2}\Big] && \text{and}\quad d=1,2,\\ &\displaystyle \tau^{1+2\gamma} &&\text{for}\quad \gamma\in \Big(0,\frac{1}{2}\Big] && \text{and}\quad d=3,\\ 
&\displaystyle \tau^{4\gamma-} &&\text{for}\quad \gamma\in \Big(\frac{1}{2},\frac{3}{4}\Big] && \text{and}\quad d=3.
\end{aligned} 
\right.
\end{align}
\end{lemma}

\begin{proof}
We first consider the case with $d=1,2$. 
For the simplicity of notation, we define $(\tilde{u}(t_n+s),\tilde{v}(t_n+s))^\top : = e^{s L}(u(t_n),v(t_n))^\top$. 
Then, by substituting $U=\Pi_{\tau^{-1}} U(t_n)$ into \eqref{eq:dt_G} we obtain
\begin{align}
&\Big\|\frac{\d}{\d s}e^{-sL}\Sigma\big(e^{sL}\Pi_{\tau^{-1}} U(t_{n})\big)\Big\|_{0}
=\left\|\begin{pmatrix}
-\sigma(\Pi_{\tau^{-1}} \tilde{u}(t_{n}+s))\\ \sigma^{\prime} \big(\Pi_{\tau^{-1}} \tilde{u}(t_{n}+s)\big)\Pi_{\tau^{-1}} \tilde{v}(t_{n}+s)
\end{pmatrix}\right\|_{0}\notag\\[2mm]
&=\|\sigma(\Pi_{\tau^{-1}} \tilde{u}(t_{n}+s))\|_{L^2}+\left\|\sigma^{\prime} \big(\Pi_{\tau^{-1}} \tilde{u}(t_{n}+s)\big)\Pi_{\tau^{-1}} \tilde{v}(t_{n}+s)\right\|_{H^{-1}}.
\end{align}
Then, using the negative norm estimates in Lemma~\ref{lem:H_estimate}, we obtain the following estimate: 
\begin{align}
\Big\|\frac{\d}{\d s}e^{-sL}\Sigma\big(e^{sL}\Pi_{\tau^{-1}} U(t_{n})\big)\Big\|_{0}\lesssim 
\tau^{2\gamma-1-}\cdot\|U(t_n)\|_{\gamma}^2\quad \quad\text{for}\quad 0<\gamma\leq \frac{1}{2}.
\end{align}
This implies that, using It\^{o}'s isometry, 
\begin{align}\label{eq:R_22}
\mathbb{E}\left[\|R_{22}(t_n)\|_{0}^{2}\right]&\leq \mathbb{E}\left[\int_{0}^{\tau}\left(\int_{0}^{s}\Big\|\frac{\d}{\d \delta}e^{-\delta L}\Sigma\big(e^{\delta L}\Pi_{\tau^{-1}} U(t_{n})\big)\Big\|_{0}\d\delta\right)^2 \d s\right]\notag\\
&\lesssim \tau^{1+4\gamma-} \cdot\mathbb{E}\left[\|U(t_n)\|_{\gamma}^4\right]\quad\quad \text{for}\quad 0<\gamma\leq \frac{1}{2},
\end{align}
Similarly, $R_{21}(t_n)$ and $R_{23}(t_n)$ can be estimated by using It\^{o}'s isometry, negative-norm estimate for the frequency localization in Lemma \ref{lem:freq_local_err}, and Bernstein's inequality in Lemma~\ref{bound}, i.e., 
\begin{align}\label{eq:R_21_R_23}
\mathbb{E}\left[\|R_{21}(t_n)\|_{0}^{2}\right]+\mathbb{E}\left[\|R_{23}(t_n)\|_{0}^{2}\right]\lesssim & \int_{0}^{\tau}\mathbb{E}\Big[\big\|\Sigma\big(e^{sL}U(t_{n})\big)-\Sigma\big(e^{sL}\Pi_{\tau^{-1}} U(t_{n})\big)\big\|_{0}^{2}\Big]\d s\notag\\
&+\mathbb{E}\Big[\big\|\Sigma(\Pi_{\tau^{-1}} U(t_n))-\Pi_{\tau^{-1}}\Sigma(\Pi_{\tau^{-1}}U(t_n))\big\|_{0}^2\Big]\cdot \tau\notag\\
\lesssim&\tau^{1+4\gamma-}\cdot\mathbb{E}[\|e^{sL}U(t_n)\|_{\gamma}^4]+\tau^{5}\cdot\mathbb{E}[\|\sigma(\Pi_{\tau^{-1}} u(t_n))\|_{H^{1}}^2]\notag\\
\lesssim& \tau^{1+4\gamma-}\cdot\mathbb{E}[\|U(t_n)\|_{\gamma}^4]+\tau^{3+2\gamma}\cdot\mathbb{E}[1+\|U(t_n)\|_{\gamma}^2] . 
\end{align}
In the case $\gamma\in(0, \frac{1}{2}]$, we obtain the first result of \eqref{eq:0_norm_R2} from \eqref{eq:R_22}--\eqref{eq:R_21_R_23} and \eqref{Lp_bound}.

In the case $d=3$, the negative-norm estimate in Lemma \ref{lem:H_estimate_3d} implies that 
\begin{align}\label{3d-1}
\left\|\sigma^{\prime} \big(\Pi_{\tau^{-1}} \tilde{u}(t_{n}+s)\big)\Pi_{\tau^{-1}} \tilde{v}(t_{n}+s)\right\|_{H^{-1}}\lesssim 
\tau^{2\gamma-\frac{3}{2}-}\cdot\|U(t_n)\|_{\gamma}^2\quad \quad\text{for}\quad 0<\gamma\leq \frac{3}{4},
\end{align}
and Bernstein's inequality implies that  
\begin{align}\label{3d-2}
\left\|\sigma^{\prime} \big(\Pi_{\tau^{-1}} \tilde{u}(t_{n}+s)\big)\Pi_{\tau^{-1}} \tilde{v}(t_{n}+s)\right\|_{H^{-1}}&\lesssim \left\|\sigma^{\prime} \big(\Pi_{\tau^{-1}} \tilde{u}(t_{n}+s)\big)\Pi_{\tau^{-1}} \tilde{v}(t_{n}+s)\right\|_{L^{2}}\notag\\
&\lesssim \tau^{\gamma-1}\|U(t_n)\|_{\gamma},\quad \quad\text{for}\quad \gamma>0 . 
\end{align}
Therefore, the remainder $R_{22}(t_n)$ can be estimated as follows (using It\^{o}'s isometry):
\begin{align}\label{eq:R_22_3d}
\mathbb{E}\left[\|R_{22}(t_n)\|_{0}^{2}\right]&\leq \mathbb{E}\left[\int_{0}^{\tau}\Big(\int_{0}^{s}\Big\|\frac{\d}{\d \delta}e^{-\delta L}\Sigma\big(e^{\delta L}\Pi_{\tau^{-1}} U(t_{n})\big)\Big\|_{0}\d\delta\Big)^2 \d s\right]\notag\\
&\lesssim \left\{\begin{aligned} 
&\tau^{1+2\gamma} \,\mathbb{E}\big[\|U(t_n)\|_{\gamma}^2\big] && \text{for} && \gamma \in \Big(0, \frac{1}{2} \Big] &&\mbox{(here \eqref{3d-2} is used)} , \\ 
&\tau^{4\gamma-} \,\mathbb{E}\big[\|U(t_n)\|_{\gamma}^4\big] && \text{for} && \gamma \in \Big(\frac12, \frac{3}{4} \Big] &&\mbox{(here \eqref{3d-1} is used)} .
\end{aligned}
\right.
\end{align} 
Similarly, for $R_{21}(t_n)$ and $R_{23}(t_n)$, it follows from the discussions in \eqref{eq:R_21_R_23} and Lemma~\ref{lem:freq_local_err} that 
\begin{align}\label{eq:R_21_R_23_3d}
&\mathbb{E}\left[\|R_{21}(t_n)\|_{0}^{2}\right]+\mathbb{E}\left[\|R_{23}(t_n)\|_{0}^{2}\right]\\
&\lesssim \tau^{3+2\gamma} \,\mathbb{E}\big[1+\big\|U(t_n)\big\|_{\gamma}^2\big]+ \left\{\begin{aligned}
&\tau^{1+2\gamma} \,\mathbb{E}\big[\|U(t_n)\|_{\gamma}^2\big] && \text{for} && \gamma \in \Big(0, \frac{1}{2} \Big],\\
&\tau^{4\gamma-} \,\mathbb{E}\big[\|U(t_n)\|_{\gamma}^4\big] && \text{for} && \gamma \in \Big(\frac12, \frac{3}{4} \Big] . 
\end{aligned}
\right.
\end{align}

Finally, combing the estimates \eqref{eq:R_22_3d} and \eqref{eq:R_21_R_23_3d}, we obtain the 3D results in Lemma \ref{lem:r_2}.
\end{proof}

Now, substituting \eqref{eq:i_1} into \eqref{eq:3-1}, we obtain
\begin{equation}\label{eq:final_U}
U(t_{n}+\tau)=e^{\tau L}U(t_{n})+e^{\tau L}\Pi_{\tau^{-1}}\Sigma\big(\Pi_{\tau^{-1}}U(t_{n})\big)\Delta_{n}W + R(t_{n}),
\end{equation}
with a remainder 
\begin{equation}\label{eq:R_tn}
R(t_{n})=R_{1}(t_{n})+R_{2}(t_{n})+I_2(t_{n}) 
\end{equation}
satisfying the following estimates (according to Lemma \ref{lem:r_1}--\ref{lem:r_2}): 
%It follows from Lemma~\ref{lem:r_1}, Lemma~\ref{lem:r_2}, and Lemma~\ref{lem:1_norm_I2} that
\begin{equation}\label{eq:estimate_R_tn}
\mathbb{E}\left[\|R(t_{n})\|_{0}^{2}\right] 
\lesssim 
\left\{ 
\begin{aligned} 
&\tau^{1+4\gamma-} && \text{for} && \gamma\in \Big(0,\frac{1}{2}\Big] 
&& \text{and}\quad d=1,2, \\[-2pt] 
&\tau^{1+2\gamma} && \text{for} && \gamma\in \Big(0,\frac{1}{2}\Big] 
&& \text{and}\quad d=3, \\[-2pt]  
&\tau^{4\gamma-} && \text{for} && \gamma\in \Big(\frac{1}{2},\frac{3}{4}\Big] 
&& \text{and}\quad d=3. 
\end{aligned}
\right.
\end{equation}
Thus, by dropping the remainder $R(t_n)$ from \eqref{eq:final_U}, we obtain the first-order scheme in \eqref{eq:filtered_1st_scheme}.

\subsection{Error estimates in Theorem \ref{thm:0_norm_error_rough_data}}\label{subsec:filtered_1st}

We compare the numerical scheme \eqref{eq:filtered_1st_scheme} with the equation \eqref{eq:final_U} satisfied by the exact solution. The difference between the two equations is written as follows:  
\begin{equation*}
U(t_{n+1})-U^{n+1}=e^{\tau L}(U(t_{n})-U^{n})+e^{\tau L}\Pi_{\tau^{-1}}\left[\Sigma(\Pi_{\tau^{-1}}U(t_n))-\Sigma(\Pi_{\tau^{-1}}U^n)\right]\Delta_n W+R(t_{n}) . 
\end{equation*}
By iterating this relation with respect to $n$, we obtain 
\begin{align}\label{eq:subtracting_equation_2}
U(t_{n+1})-U^{n+1}=e^{(n+1)\tau L}(U(t_{0})-U^0)+\sum_{j=0}^{n}e^{(n-j)\tau L}\mathcal{L}^{j},
\end{align}
where 
$$
\mathcal{L}^{j}=e^{\tau L}\Pi_{\tau^{-1}}\left[\Sigma(\Pi_{\tau^{-1}}U(t_j))-\Sigma(\Pi_{\tau^{-1}}U^j)\right]\Delta_j W+R(t_{j}).
$$
Then, by applying the $\|\cdot\|_0$ norm to both sides of \eqref{eq:subtracting_equation_2} and using relation $U(t_0)=U^0$, we have 
\begin{align}\label{eq:0_norm_error_estimate}
\big\|U(t_{n+1})-U^{n+1}\big\|_{0}^{2}=
%&\big\|e^{(n+1)\tau L}(U(t_{0})-U^0)\big\|_{0}^{2}+2\sum_{j=0}^{n}\big\langle e^{(n+1)\tau L}(U(t_{0})-U^0),e^{(n-j)\tau L}\mathcal{L}^{j} \big\rangle_0\notag\\
&\sum_{j,k}^{n}\big\langle e^{(n-j)\tau L}\mathcal{L}^{j}, e^{(n-k)\tau L}\mathcal{L}^{k} \big\rangle_0.
\end{align}
The expectation of the right-hand side of \eqref{eq:0_norm_error_estimate} can be estimated by using It\^{o}'s isometry and the independency of the random variables $\mathcal{L}^{j}$, $j=1,2,\dots$, i.e., 
\begin{align}
\mathbb{E}\big[\|U(t_{n+1})-U^{n+1}\|_{0}^{2}\big] 
&=
%\mathbb{E}\big[\|e^{(n+1)\tau L}(U(t_{0})-U^0)\|_{0}^{2}\big]+
\sum_{j=0}^n\mathbb{E}\big[\|e^{(n-j)\tau L}\mathcal{L}^{j}\|_{0}^2\big]\notag\\
&\lesssim 
%\mathbb{E}\big[\|U(t_{0})-U^0\|_{0}^{2}\big]+
\sum_{j=0}^n\tau\mathbb{E}\big[\|\Sigma(\Pi_{\tau^{-1}}U(t_j))-\Sigma(\Pi_{\tau^{-1}}U^j)\|_{0}^2\big] +\sum_{j=0}^n\mathbb{E}\big[\|R(t_j)\|_{0}^2\big].
\end{align}
Since
\begin{align*}
\left\|\Sigma(\Pi_{\tau^{-1}}U(t_j))-\Sigma(\Pi_{\tau^{-1}}U^j)\right\|_{0}&= \left\|\sigma(\Pi_{\tau^{-1}}u(t_j))-\sigma(\Pi_{\tau^{-1}}u^j)\right\|_{H^{-1}}\\
&\lesssim \left\|\sigma(\Pi_{\tau^{-1}}u(t_j))-\sigma(\Pi_{\tau^{-1}}u^j)\right\|_{L^2}\lesssim \left\|U(t_j)-U^j\right\|_{0},
\end{align*}
it follows that Lemma~\ref{lem:r_1}--\ref{lem:r_2} that
\begin{align}
\mathbb{E}\big[\|U(t_{n+1})-U^{n+1}\|_{0}^{2}\big]
\lesssim &
%\mathbb{E}\left[\left\|U(t_{0})-U^0\right\|_{0}^{2}\right]+
\sum_{j=0}^n\tau\mathbb{E}\big[\|U(t_j)-U^j\|_{0}^2\big]\notag\\
&+\left\{\begin{aligned} 
&\sum_{j=0}^n\tau^{1+4\gamma-} &&\text{for}&& \gamma\in\Big(0,\frac{1}{2}\Big]&& \text{and}&& d=1,2,\\[-5pt]
&\sum_{j=0}^n\tau^{1+2\gamma} &&\text{for}&& \gamma\in\Big(0,\frac{1}{2}\Big]&& \text{and}&& d=3,\\[-5pt]
&\sum_{j=0}^n\tau^{4\gamma-} &&\text{for}&& \gamma\in\Big(\frac{1}{2},\frac{3}{4}\Big]&& \text{and}&& d=3,
\end{aligned} 
\right.
\end{align}
under the condition $U^0\in H^{\gamma}(\mathcal{O})\times H^{\gamma-1}(\mathcal{O})$.
By using the discrete Gronwall's inequality, we obtain the result of Theorem \ref{thm:0_norm_error_rough_data}.\hfill\qed

\section{The algorithm implementations and numerical experiments}\label{section:numerical_test}

In this section, we discuss the implementation of the proposed low-regularity integrator in \eqref{eq:filtered_1st_scheme} for solving the stochastic wave equation in \eqref{system}, and provide numerical results which support the theoretical analysis and demonstrate the advantages of the proposed method. 

\subsection{Implementation of the algorithm}\label{subsec:fully_discrete}

Firstly, we discuss the implementation of the algorithm with a specially designed Fourier spectral method for spatial discretization, called high-frequency recovering method. This method is based on low- and high-frequency decomposition of the semi-discrete numerical solution $U^{n+1}$ through rewriting \eqref{eq:filtered_1st_scheme} as follows: 
\begin{align*}
\Pi_{\tau^{-1}}U^{n+1}&=e^{\tau L}\Pi_{\tau^{-1}}U^{n}+\tau e^{\tau L}\Pi_{\tau^{-1}}F(\Pi_{\tau^{-1}}U^{n})+e^{\tau L}\Pi_{\tau^{-1}}\Sigma\left(\Pi_{\tau^{-1}}U^{n}\right)\Delta_n W ,\\
\Pi_{>\tau^{-1}}U^{n+1}&=e^{\tau L}\Pi_{>\tau^{-1}}U^{n},
\end{align*}
where the low- and high-frequency parts of $U^{n+1}$ can be computed independently of each other, and the stochastic noise appears only in the low-frequency part. The high-frequency part of the solution at any given time level $t_{n+1}$ can be solved directly as follows (without relying on time steppings): 
$$
\Pi_{>\tau^{-1}}U^{n+1}=e^{(n+1)\tau L}\Pi_{>\tau^{-1}}U^0 .
$$ 
This inspires us to discretize the original problem with the Fourier spectral method and the following high-frequency recovering scheme: 
%\textbf{High-frequency recovering first-order fully-discrete scheme:}
\begin{subequations}\label{eq:1st_HRLRI}
\begin{align}
U^{n+1}_N&=\Pi_{N}U^{n+1}_N+\Pi_{(N,N^{\alpha}]}U^{n+1}_N,\label{eq:1st_HRLRI_1}\\[2mm]
\Pi_{N}U^{n+1}_N&=e^{\tau L}\Pi_{N}U^{n}_N+\tau e^{\tau L}I_{N}F(\Pi_{N}U^{n}_N)+e^{\tau L}I_{N}\Sigma\left(\Pi_{N}U^{n}_N\right)\Delta_n W,\label{eq:1st_HRLRI_3} \\[2mm]
\Pi_{(N,N^{\alpha}]}U^{n+1}_N&=e^{(n+1)\tau L}\Pi_{(N,N^{\alpha}]}U^{0},\label{eq:1st_HRLRI_2} 
\end{align}
\end{subequations}
where $N=O(\tau^{-1})$, and the high-frequency part $\Pi_{(N,N^{\alpha}]}U^{n+1}_N$ computed by \eqref{eq:1st_HRLRI_2} aims at improving the accuracy of spatial discretization without essentially increasing the computational cost, with $\alpha\geq 1$ being a parameter to be chosen in order to balance the computational cost between the time-stepping scheme in \eqref{eq:1st_HRLRI_3} and the computation of the high-frequency part in \eqref{eq:1st_HRLRI_2}. The implementation of \eqref{eq:1st_HRLRI} is presented in the following table (Algorithm 1): 

%We first give Algorithm 1 corresponding to the high-frequency recovery first-order scheme \eqref{eq:1st_HRLRI} as follows:
\begin{algorithm}[htb]
\caption{High-frequency recovered low-regularity integrator \eqref{eq:1st_HRLRI}.}
\LinesNumbered
\KwIn{Initial value $U^0$, final time $T$}
%Set $\tau=T/M$, where $M$ is a positive integer. \\
Take $\alpha\geq 1$, and decompose the initial state $U^0=\Pi_{N}U^0+\Pi_{(N,N^\alpha]}U^0+\Pi_{>N^{\alpha}}U^0$\\

\For{$n=0$ to $T/\tau-1$}{
Generate the random variable $\Delta_n W\sim N(0,\tau)$\\
Calculate $I_N \Sigma(\Pi_N U^n_N)$ by using the FFT\\
Determine $\Pi_N U^{n+1}$ from $\Pi_N U^{n}$ and $I_N \Sigma(\Pi_N U^n_N)$ by \eqref{eq:1st_HRLRI_3}
}
Implement the high-frequency recovery process by taking 
$\Pi_{(N,N^{\alpha}]}U^{T/\tau}=e^{TL}\Pi_{(N,N^\alpha]}U^0.$\\
\KwOut{Final state $U^{T/\tau}=\Pi_N U^{T/\tau}+\Pi_{(N,N^{\alpha}]}U^{T/\tau}$.}
\end{algorithm}

Since the nonlinear term $I_N \Sigma(\Pi_N U^n_N)$ can be computed with the Fast Fourier Transform (FFT), the computational cost at every time level is $O(N^d\log(N)^d)$. Therefore, the total cost for computing the low-frequency part at time $T$ is $O(N^{d}\log(N)^d T/\tau )$. In contrast, the high-frequency part of the numerical solution needs not be computed every time level. We only need to compute $e^{TL}\Pi_{(N,N^\alpha]}U^0$ once to recover the high-frequency part of $U^{T/\tau}_{N}$. Therefore, the cost of computing the high-frequency part at time $T$ is $O(N^{\alpha d})$. Under the stepsize condition $\tau= O( N^{-1} )$, one can choose either $\alpha=1+1/d$ to balance the cost for computing the low- and high-frequency parts, in order to improve the accuracy of spatial discretization through adding the high-frequency part without essentially increasing the computational cost. 
%and temporal errors for good approximation accuracy for the total error when $\gamma\leq 1/2$ and $d=1,2$, or $\alpha=1+\frac1d$ for all $\gamma>0$ and $d=1,2,3$ to make the cost of the high-frequency recovering process comparable to that of calculating the low-frequency part of the process.

\subsection{The stochastic nonlinear wave equation in one dimension}

\begin{example}\label{Example1}\upshape
We first consider the stochastic nonlinear wave equation \eqref{system} with $\sigma(u) = 16\sin(u)$ under the following piecewise smooth discontinuous initial conditions:
\begin{align}\label{1d-initial-value}
\big(u^{0}(x), v^{0}(x)\big) 
=\left\{
\begin{aligned}
&(5, 0) && \text{for}\,\,\, x \in [0.3, 0.425], \\[-1pt]
&(2.5, 0) && \text{for}\,\,\, x \in [0.575, 0.7], \\[-1pt]
&(0, 0) && \text{elsewhere}.
\end{aligned}
\right.
\end{align}
We numerically solve this problem using the proposed method in \eqref{eq:1st_HRLRI} with parameters $\alpha = 2$, $N = 2^{10}$, and $\tau = \frac14 N^{-1}$. In Figure \ref{fig:4-1-1} (a) and (b), we present the evolution of the numerical solution for $t \in [0, T]$ across two different sample paths. The numerical results indicate that the evolution of discontinuous solutions can vary significantly due to the presence of stochastic noise in the equation. Our method effectively captures the evolution of these discontinuous solutions across different sample paths. 

For the two sample paths corresponding to Figures \ref{fig:4-1-1} (c) and (d), we compare the numerical solutions at $T = 0.25$ computed by the proposed method \eqref{eq:1st_HRLRI} (referred to as \texttt{HR-LRI}) with those obtained using the well-established semi-implicit Euler-Maruyama method (SEM) given by 
\begin{align}\label{eq:SEM}
U^{n+1} = U^n + \tau L U^{n+1} + I_N \Sigma(U^n)\Delta_n W,
\end{align}
and the stochastic trigonometric method (STM), expressed as 
\begin{align}\label{eq:STM}
U^{n+1} = e^{\tau L}U^n + e^{\tau L}I_N \Sigma(U^n)\Delta_n W.
\end{align}
The stepsize of time discretization and the degrees of freedoms for spatial discretization are set to $\tau = \frac14 N^{-1}$ and $N = 2^7$, respectively. The reference solution is computed using \eqref{eq:1st_HRLRI} with $\tau = \frac14  N^{-1} = 2^{-14}$. The numerical results demonstrate that the proposed method achieves high accuracy and effectively reduces the spurious oscillations caused by the discontinuities. 

\begin{figure}[htbp!]
\centering
\hspace{-10pt}
\subfigure[Propagation of $u(t,x)$ in Path 1]{\includegraphics[width=7.0cm,height=5.1cm]{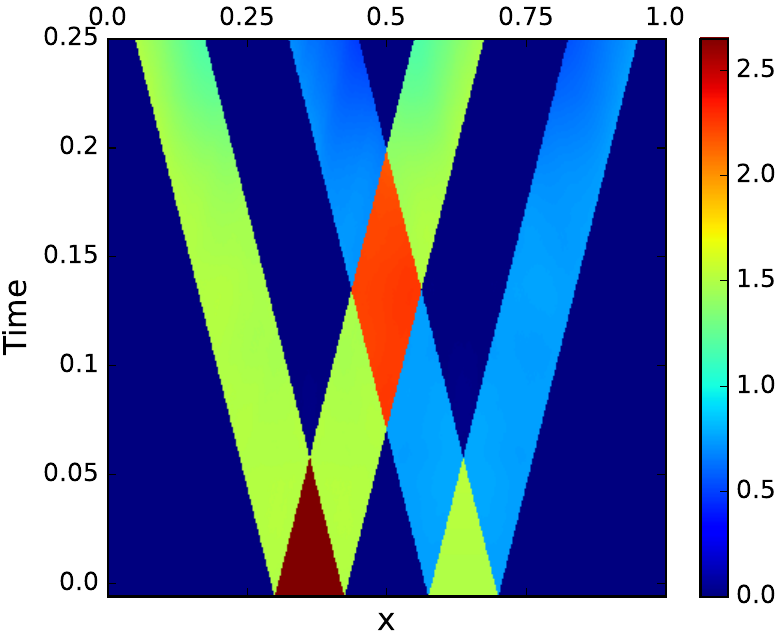}}
\qquad
\subfigure[Propagation of $u(t,x)$ for Path 2]{\includegraphics[width=7.0cm,height=5.1cm]{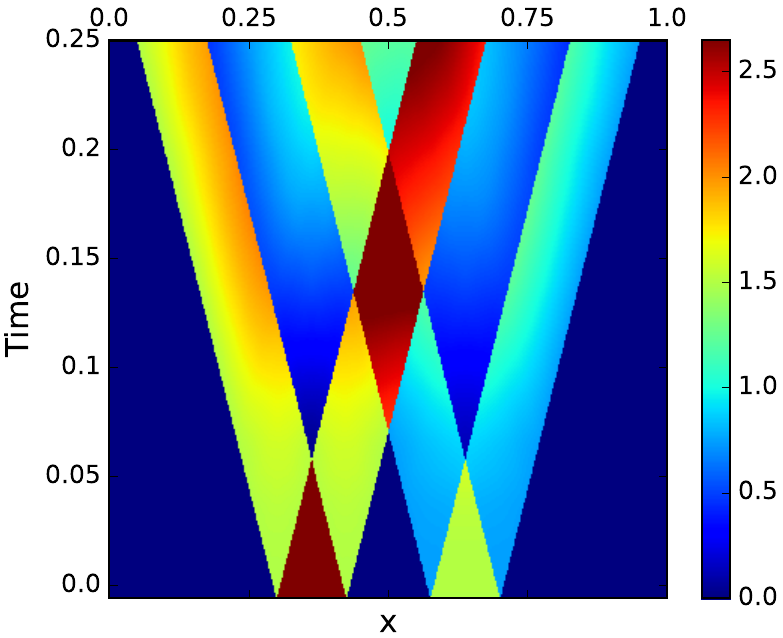}}
\vfill
\centering
\subfigure[Numerical results of $u(T,x)$ for Path 1]{\includegraphics[width=7.0cm,height=6.0cm]{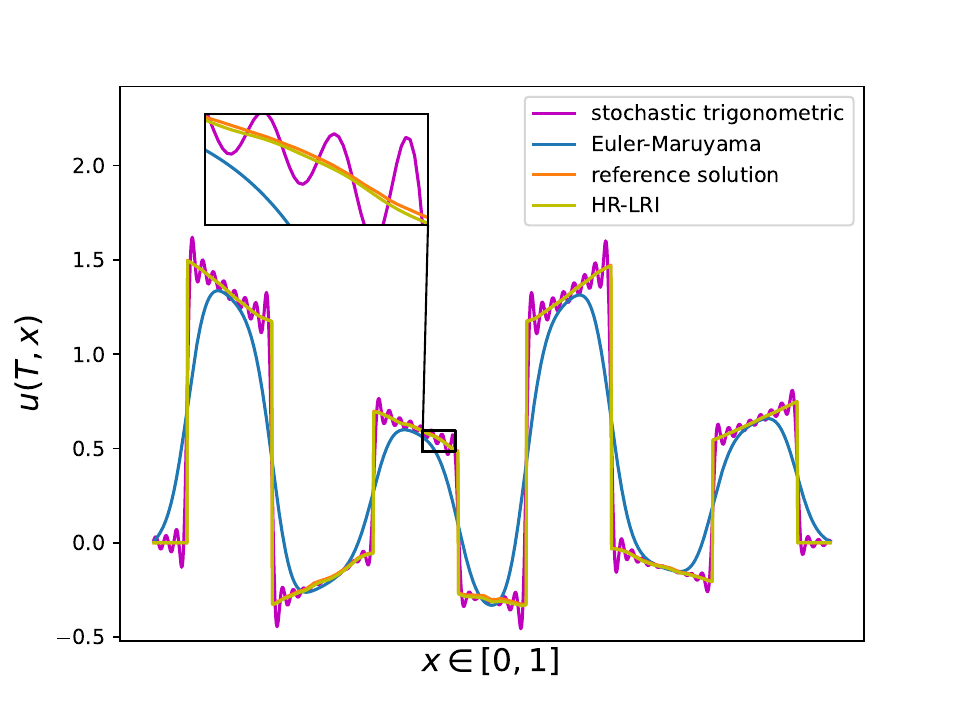}}
\qquad
\subfigure[Numerical results of $u(T,x)$ in Path 2]{\includegraphics[width=7.0cm,height=6.0cm]{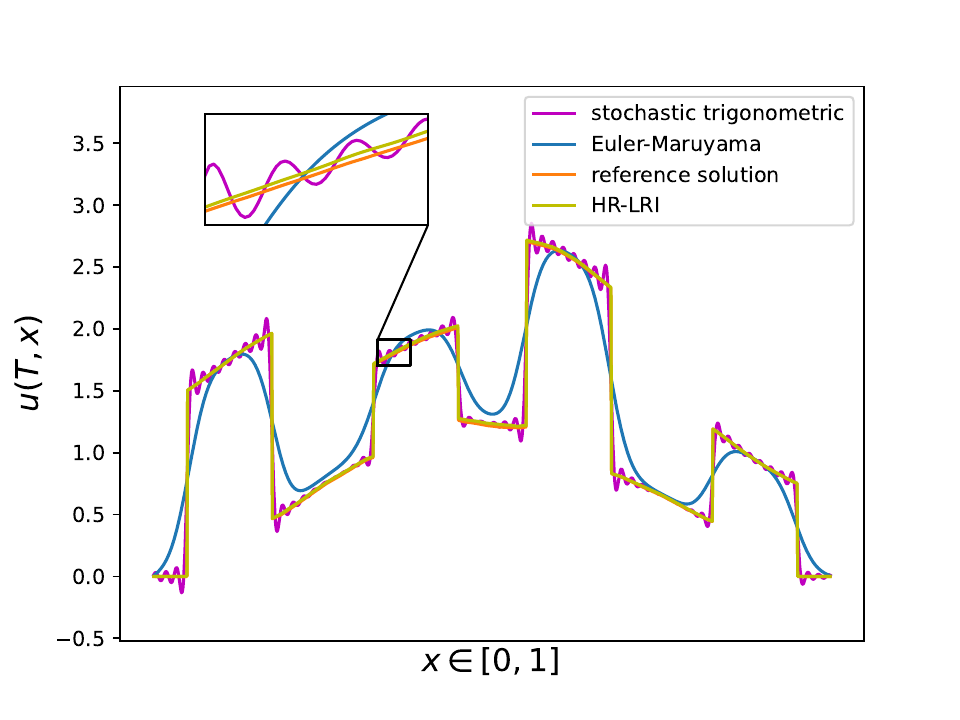}}
\caption{Numerical solutions of the 1D problem (Example \ref{Example1}). }
\label{fig:4-1-1}
\end{figure}

In Figure \ref{fig:4-2}, we compare the mean-square errors in the $L^2(\mathcal{O}) \times H^{-1}(\mathcal{O})$ norm at time $T = 0.25$ for numerical solutions computed by several methods with $\tau = N^{-1}/4$. The reference solution is obtained using the proposed method \eqref{eq:1st_HRLRI} with a sufficiently small time step size of $\tau = N^{-1}/4 = 2^{-14}$. Additionally, the expectations of the numerical errors are estimated using the Monte Carlo method with $1000$ samples. The numerical results indicate that the proposed method achieves nearly first-order convergence, which is consistent with the theoretical result proved in Theorem\ref{thm:0_norm_error_rough_data}. In contrast, classical methods experience a significant reduction in convergence due to the roughness of the exact solution. Figure~\ref{fig:4-2} (b) demonstrates that the proposed method offers considerably higher accuracy than the classical methods, even when operating within the same computational time frame.

\begin{figure}[htbp!]
\centering
\subfigure[$L^{2}(\Omega)\times H^{-1}(\Omega)$ error versus $\tau$]{\includegraphics[width=7.0cm,height=6.0cm]{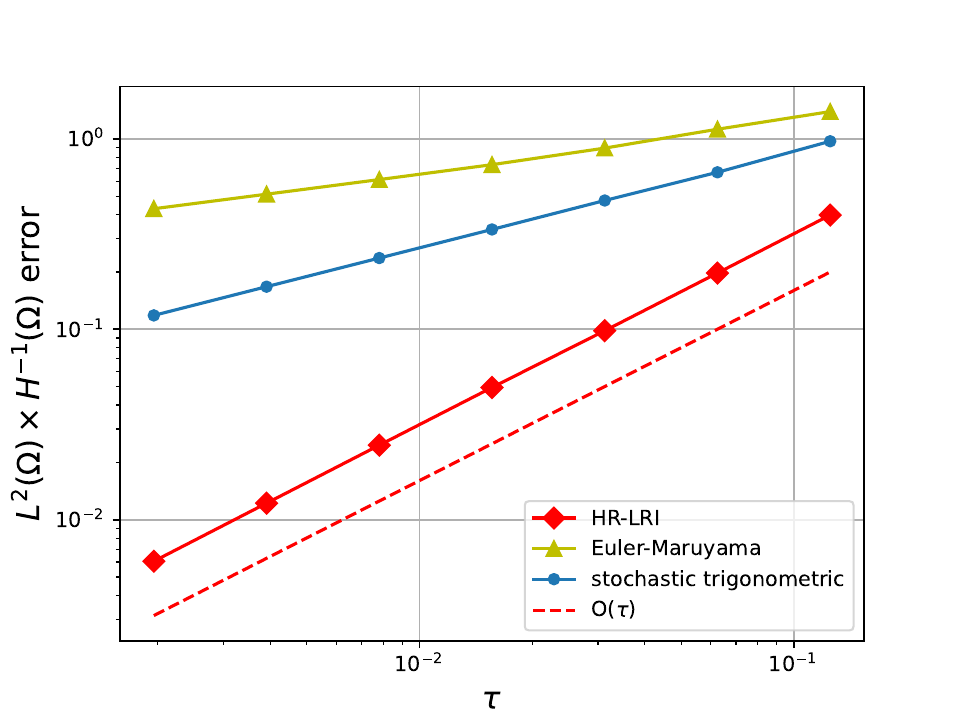}}
\qquad
\subfigure[$L^{2}(\Omega)\times H^{-1}(\Omega)$ error versus CPU time]{\includegraphics[width=7.0cm,height=6.0cm]{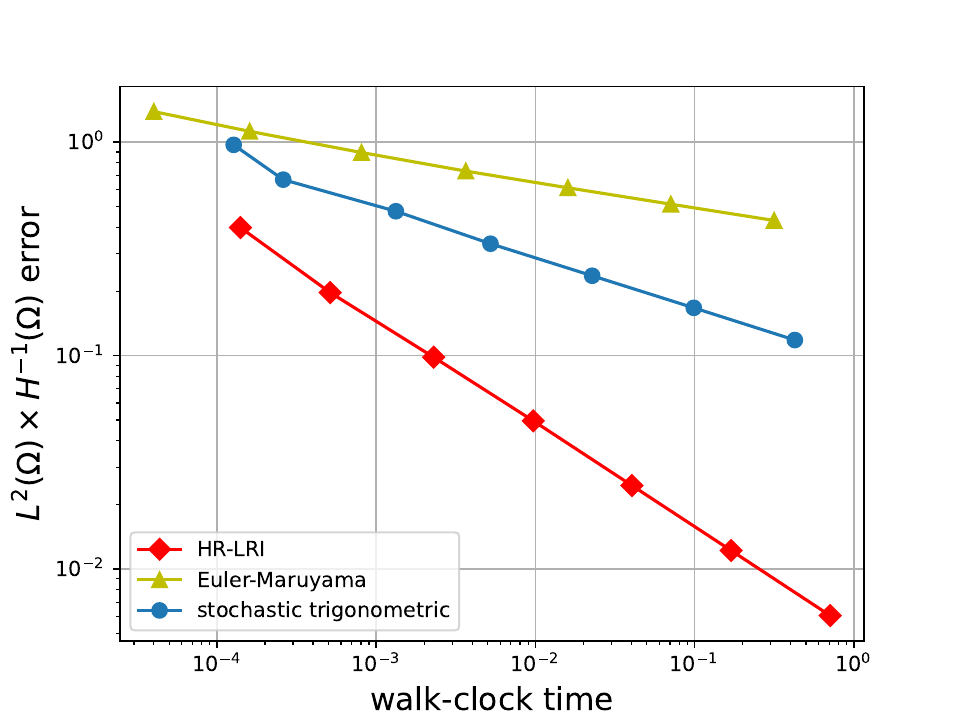}}
\caption{Errors of the numerical solutions by several methods (Example \ref{Example1}).}
\label{fig:4-2}
\end{figure}

\end{example}

\pagebreak
\,
\pagebreak

\begin{figure}[!htbp]
	\centering
	\subfigure[Solution in $H^{\frac12} \times H^{-\frac12}$ (one sample path)]{\includegraphics[width=7.1cm,height=6.0cm]{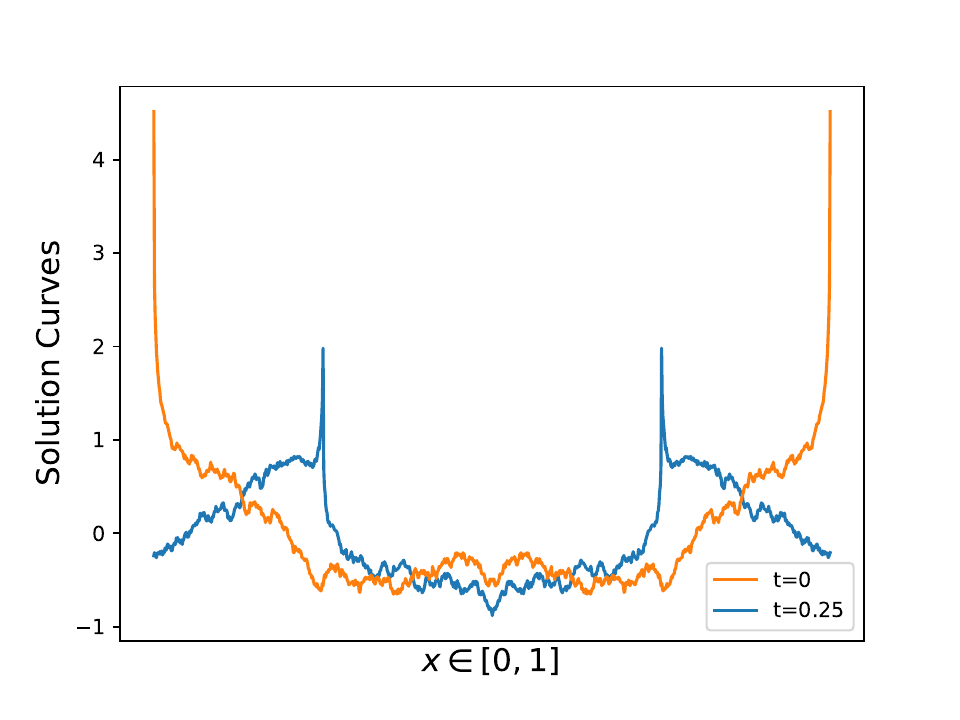}}
	\qquad
	\subfigure[Solution in $H^{4} \times H^{3}$ (one sample path)]{\includegraphics[width=7.1cm,height=6.0cm]{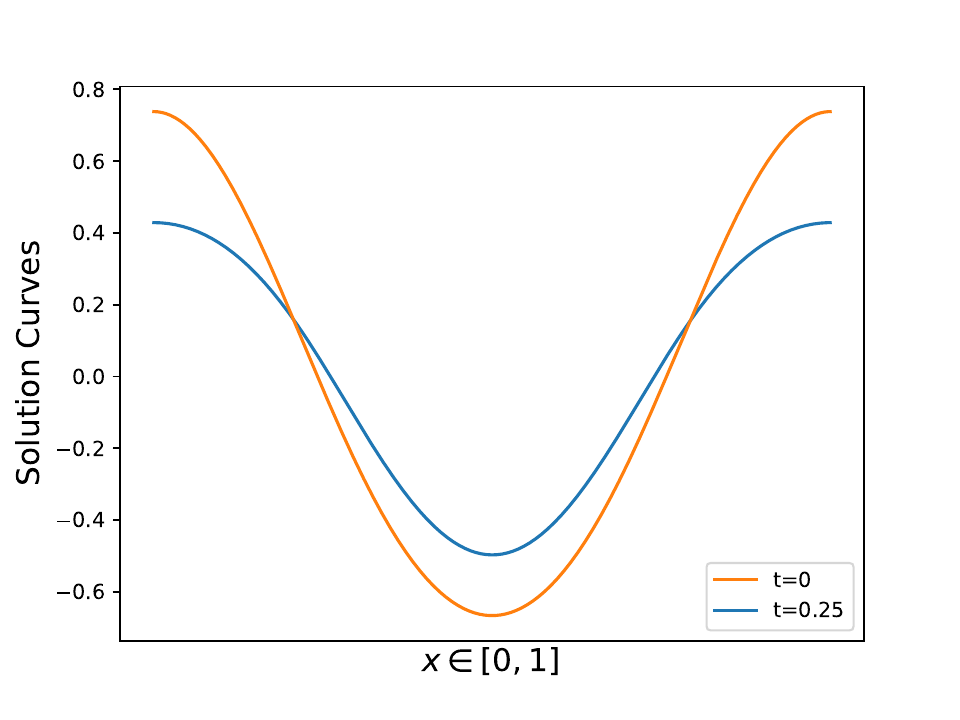}}
	\caption{Numerical solutions in $H^{\frac12} \times H^{-\frac12}$ and $H^{4} \times H^{3}$ (Example \ref{Example2}).}
	\label{fig:4-3-0}
\end{figure}

\begin{figure}[!htbp]
\centering
\subfigure[$L^{2}(\Omega)\times H^{-1}(\Omega)$ error vs $\tau$ for $\gamma=\frac{1}{2}$]{\includegraphics[width=7.0cm,height=6.0cm]{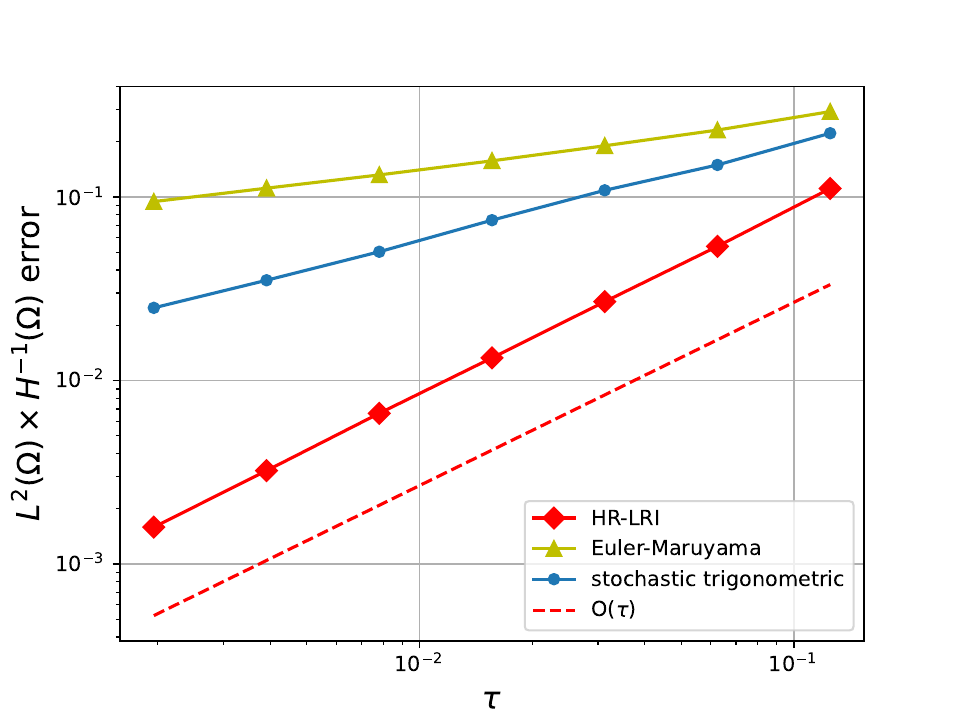}}
\qquad
\subfigure[$L^{2}(\Omega)\times H^{-1}(\Omega)$ error vs CPU time for $\gamma=\frac{1}{2}$]{\includegraphics[width=7.0cm,height=6.0cm]{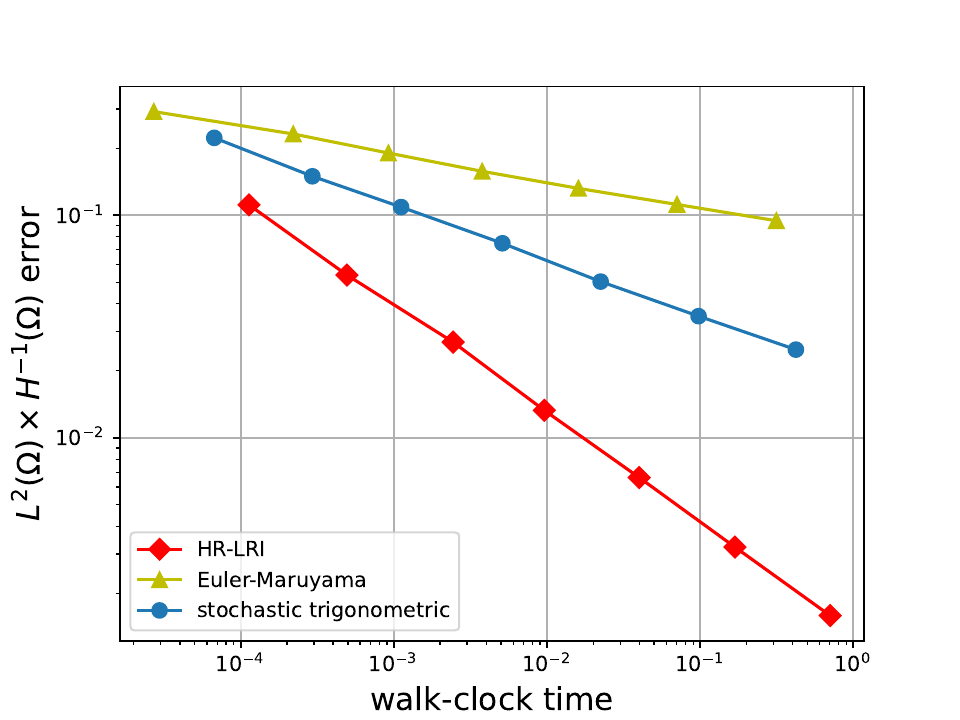}}
\vfill
\centering
\subfigure[$L^{2}(\Omega)\times H^{-1}(\Omega)$ error vs $\tau$ for $\gamma=4$]{\includegraphics[width=7.0cm,height=6.0cm]{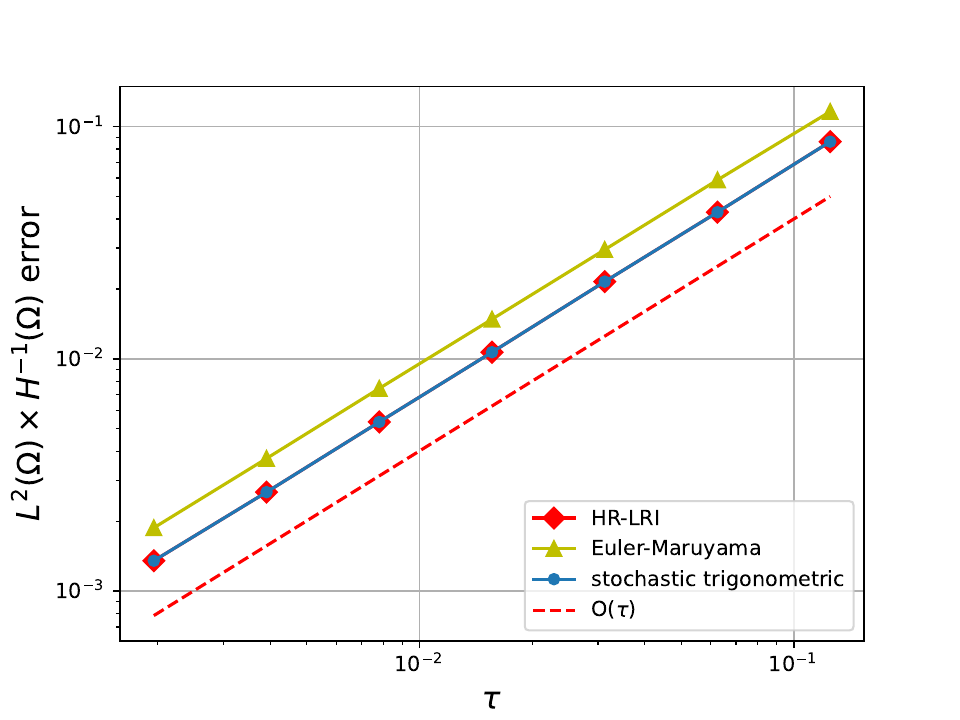}}
\qquad
\subfigure[$L^{2}(\Omega)\times H^{-1}(\Omega)$ error vs CPU time for $\gamma=4$]{\includegraphics[width=7.0cm,height=6.0cm]{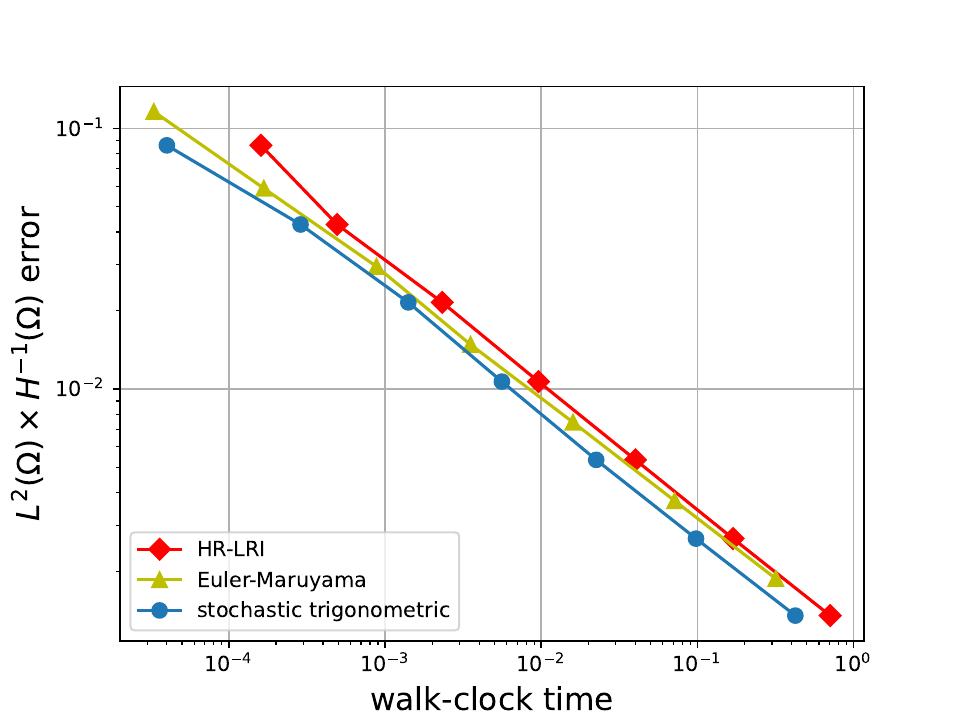}}
\caption{Errors of the numerical solutions by several methods (Example \ref{Example2}).}
\label{fig:4-3}
\end{figure}

\begin{example}\label{Example2}
We consider the stochastic nonlinear wave equation \eqref{system} with $\sigma(u)=16\sin(u)$ for the following initial state in $H^{\gamma}\times H^{\gamma-1}$: 
\begin{equation}\label{eq:rough_initial_1d}
(u^{0}(x),v^{0}(x))=\Big(\sum_{k\in\mathbb{Z}}a_u(k)e^{ikx},\;\sum_{k\in\mathbb{Z}}a_v(k)e^{ikx} \Big),
\end{equation}
where
\begin{align*}
\left\{
\begin{array}{ll}
 a_u(k)=a_u(-k)=\frac{1}{2}\text{rand}(0,1)|k|^{-\gamma-0.51},\\[2mm]
 a_v(k)=a_v(-k)=\frac{1}{2}\text{rand}(0,1)|k|^{-\gamma+0.49}.
\end{array}
\right.
\end{align*} 
%and $C_u$ and $C_v$ are constants such that $\|u^{0}\|_{H^{\gamma}}=\|v^{0}\|_{H^{\gamma-1}}=1$. 
We solve this problem numerically using the method proposed in \eqref{eq:1st_HRLRI}, with parameters $\alpha = 2$, $N = 2^{10}$, and $\tau = \frac{1}{4}N^{-1}$. In Figure \ref{fig:4-3-0}, we present the graph of the solution in one sample path, with $\gamma = 0.5$ and $\gamma = 4$, respectively. These graphs illustrate the evolution of the solution under different levels of regularity.

Figure \ref{fig:4-3} illustrates the errors in numerical solutions computed by various methods for different values of $\gamma$, using stepsize $\tau = N^{-1}/4$. The reference solution is provided by the proposed method \eqref{eq:1st_HRLRI} with stepsize $\tau = N^{-1}/4 = 2^{-14}$ and $\alpha = 2$. A Monte Carlo simulation with 1000 samples was conducted. The numerical results show that the proposed HR-LRI method has higher convergence rate than other methods in computing rough solutions in $H^{1/2}\times H^{-1/2}$, and the same convergence rate as other methods in computing smooth solutions in $H^4\times H^{3}$.

\end{example}

\subsection{The stochastic nonlinear wave equation in two dimensions}
\begin{example}\label{Example3} 
In this example, we present numerical results for the stochastic nonlinear wave equation in two dimensions. 
We first consider the problem with $\sigma(u) = 16\sin(u)$, subject to the following piecewise smooth, discontinuous initial condition:
\begin{align}\label{2d-initial-value}
	\big(u^{0}(x),v^{0}(x)\big)=\left\{
	\begin{array}{ll}
		{\displaystyle \left(0.5,0\right)},\quad &\text{for } {\displaystyle x\in \big [0.375,0.625\big]^{2}},\\[2.5mm]
		{\displaystyle \left(0,0\right)},\quad &\text{else where. }
	\end{array}
	\right. 
\end{align}		
We solve the problem using the semi-implicit Euler-Maruyama method in \eqref{eq:SEM} and the proposed method in \eqref{eq:1st_HRLRI} with $\alpha = 3/2$. The numerical solutions for two different sample paths are plotted in Figure~\ref{fig:4-4-1}, where we choose the time step $\tau = N^{-1}/4 = 2^{-8}$. The results demonstrate that the proposed method offers a significant advantage in accurately capturing the sharp interface of the discontinuous solution.
\begin{figure}[!htbp]
\centering
\subfigure[Semi-implicit Euler-Maruyama for Path 1]{\includegraphics[width=6.8cm,height=5.0cm]{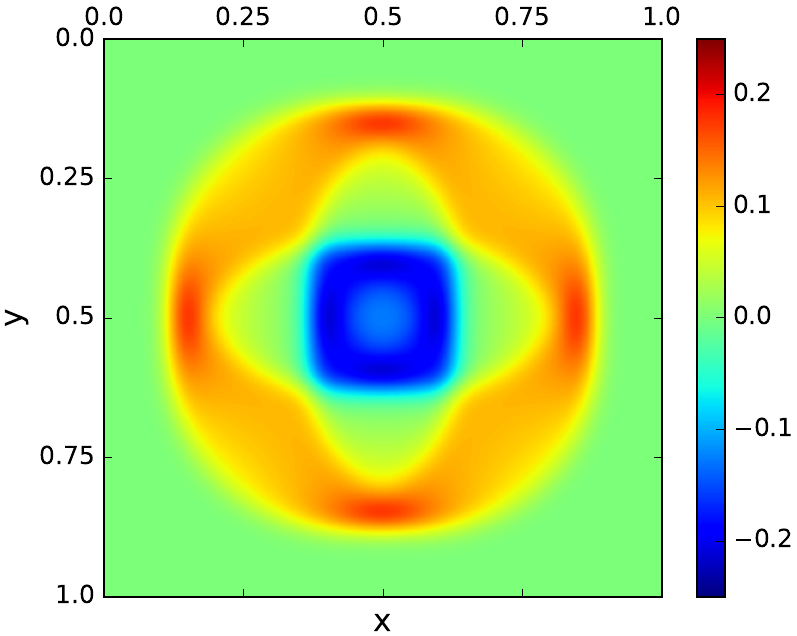}}
\quad
\subfigure[The proposed method for Path 1]{\includegraphics[width=6.8cm,height=5.0cm]{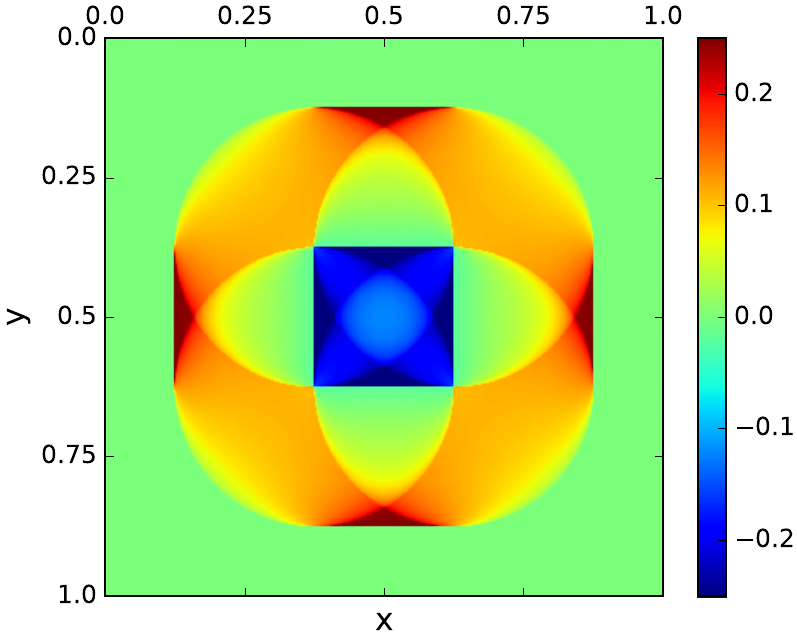}}
\vspace{10pt}
\vfill
\centering
\subfigure[Semi-implicit Euler-Maruyama for Path 2]{\includegraphics[width=6.8cm,height=5.0cm]{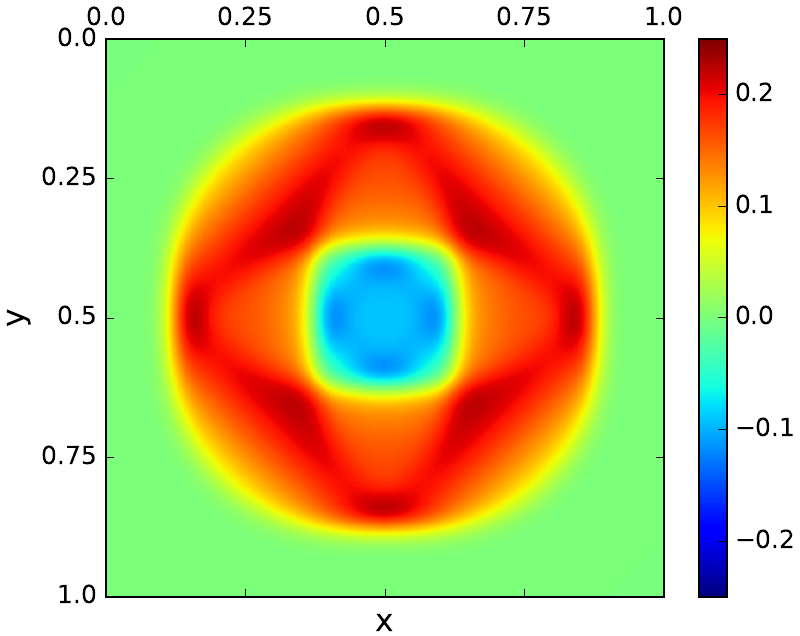}}
\quad
\subfigure[The proposed method for Path 2]{\includegraphics[width=6.8cm,height=5.0cm]{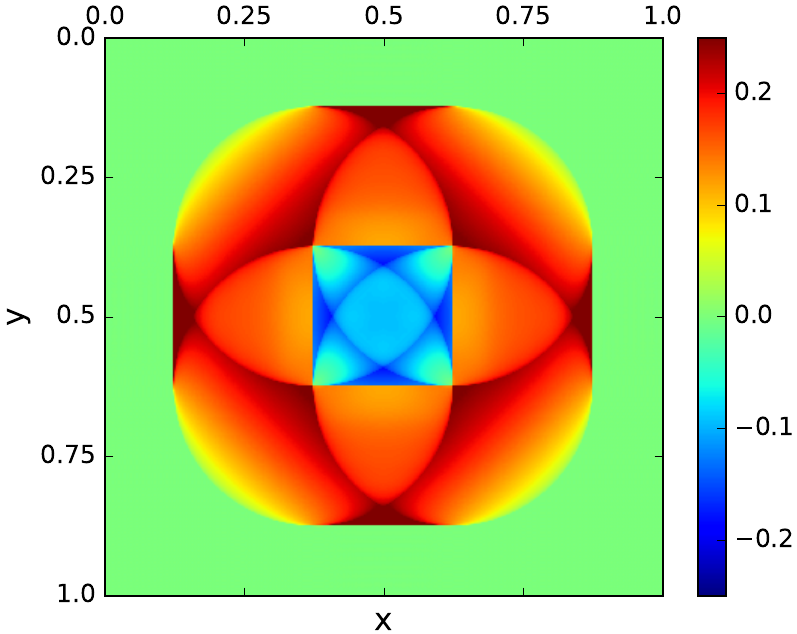}}
\caption{Numerical solutions of the 2D problem (Example \ref{Example3}).} 
\label{fig:4-4-1} 
\end{figure}

\begin{figure}[htbp!]
\centering\vspace{-15pt}
\subfigure[$L^{2}(\Omega)\times H^{-1}(\Omega)$ error versus $\tau$]{\includegraphics[width=6.8cm,height=5.8cm]{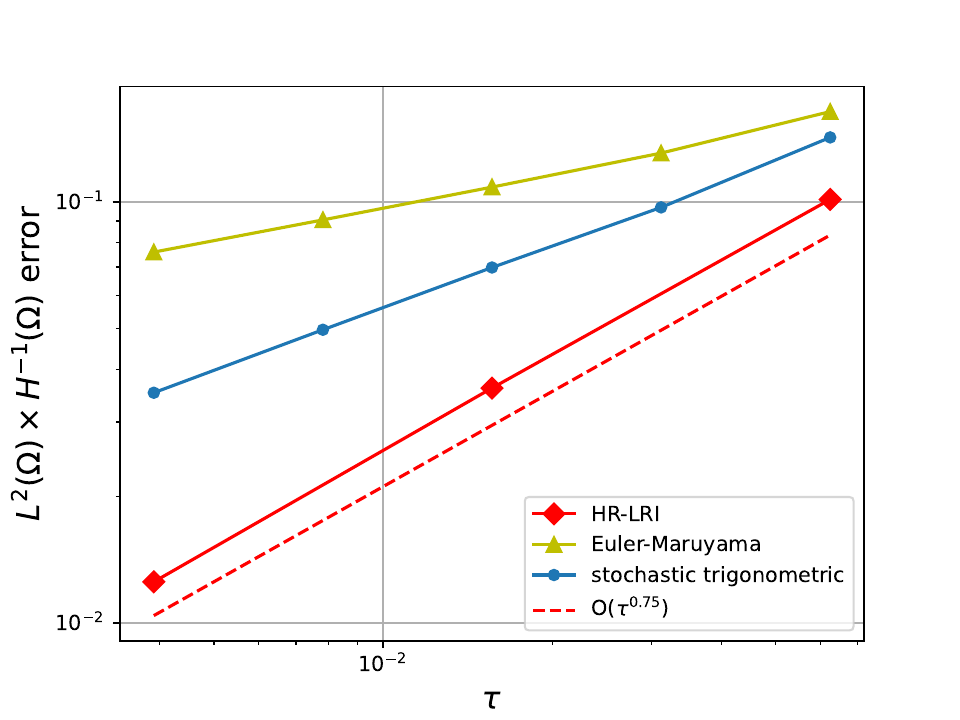}}
\quad
\subfigure[$L^{2}(\Omega)\times H^{-1}(\Omega)$ error versus CPU time]{\includegraphics[width=6.8cm,height=5.8cm]{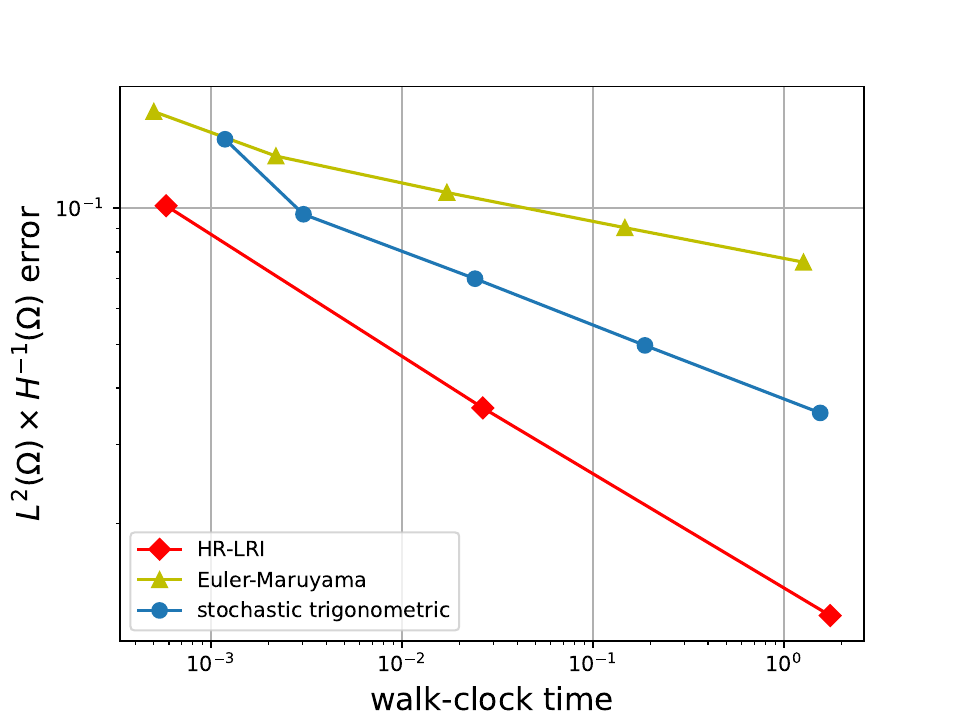}}
\caption{Errors of the numerical solutions by several methods (Example \ref{Example3}).}
\label{fig:4-4}
\end{figure}

As in previous comparisons, we evaluate the mean-square errors in the $L^2(\mathcal{O})\times H^{-1}(\mathcal{O})$ norm at time  $T=0.25$, for the numerical solutions obtained using several methods with $\tau=N^{-1}/4$. The reference solution is computed using the proposed method in \eqref{eq:1st_HRLRI} with $\tau=N^{-1}/4=2^{-10}$ and $\alpha=3/2$. The expectation of the errors are computed over 1000 sample paths and prseented in Figure \ref{fig:4-4}, where the numerical results are consistent with the convergence rate proved in Theorem \ref{thm:0_norm_error_rough_data} and demonstrating that the proposed method retains an advantage in the two-dimensional case.

\end{example}

\begin{example}\label{Example4}
We consider the two-dimensional stochastic nonlinear wave equation with \(\sigma(u) = 16 \sin(u)\), with a more general initial state \((u^{0}, v^{0}) \in H^{\gamma} \times H^{\gamma - 1}\) given by 
\begin{equation}\label{eq:rough_initial_2d}
	(u^{0}(x), v^{0}(x)) = \left( \sum_{k,l \in \mathbb{Z}} a_u(k) b_u(l) e^{i(kx_1 + lx_2)}, \; \sum_{k,l \in \mathbb{Z}} a_v(k) b_v(l) e^{i(kx_1 + lx_2)} \right),
\end{equation}
where
\begin{align*}
	\left\{
	\begin{array}{ll}
		a_u(k)=a_u(-k) = \frac{1}{2}\text{rand}(0,1) |k|^{-\gamma - 0.51}, \quad & b_u(l)= b_u(-l)= \frac{1}{2}\text{rand}(0,1) |l|^{-\gamma - 0.51}, \\[2mm]
		a_v(k)=a_v(-k) = \frac{1}{2}\text{rand}(0,1) |k|^{-\gamma + 0.49}, \quad & b_v(l)=b_v(-l) = \frac{1}{2}\text{rand}(0,1) |l|^{-\gamma + 0.49}.
	\end{array}
	\right.
\end{align*}
%and \(C_u\) and \(C_v\) are constants such that \(\|u^{0}\|_{H^{\gamma}} = \|v^{0}\|_{H^{\gamma - 1}} = 1\).

For $\gamma = \frac{1}{2}$ and $\gamma = 4$ (corresponding to initial states in $H^{\frac{1}{2}} \times H^{-\frac{1}{2}}$ and $H^{4} \times H^{3}$, respectively), we solve the problem using the method proposed in \eqref{eq:1st_HRLRI} with $\alpha = \frac{3}{2}$ and $\tau = \frac{1}{4}N^{-1} = 2^{-8}$. In Figure~\ref{fig:4-5-0}, we present the numerical solutions for different regularities. The results demonstrate the propagation of rough waves governed by the semilinear wave equation.

Figure~\ref{fig:4-5} illustrates the rates of convergence of the mean-square $L^2(\mathcal{O}) \times H^{-1}(\mathcal{O})$ error for various numerical methods with the CFL condition $\tau = N^{-1}/4$. In the proposed method \eqref{eq:1st_HRLRI}, we choose $\alpha = 3/2$ to balance the computational cost between high and low-frequency components, as discussed at the end of Section \ref{subsec:fully_discrete}. The reference solution is obtained using the same method \eqref{eq:1st_HRLRI}, with $\tau = N^{-1}/4 = 2^{-10}$, and the expectation is computed by averaging over 1000 samples. The numerical results show that the proposed HR-LRI method has higher convergence rate in computing rough solutions in $H^{1/2}\times H^{-1/2}$, and the same convergence rate as other methods in computing smooth solutions in $H^4\times H^{3}$.

\begin{figure}[!htbp]
	\centering
	\subfigure[$u_0\in H^{\frac12}\times H^{-\frac12}$ in one sample path]{\includegraphics[width=7.0cm,height=5.3cm]{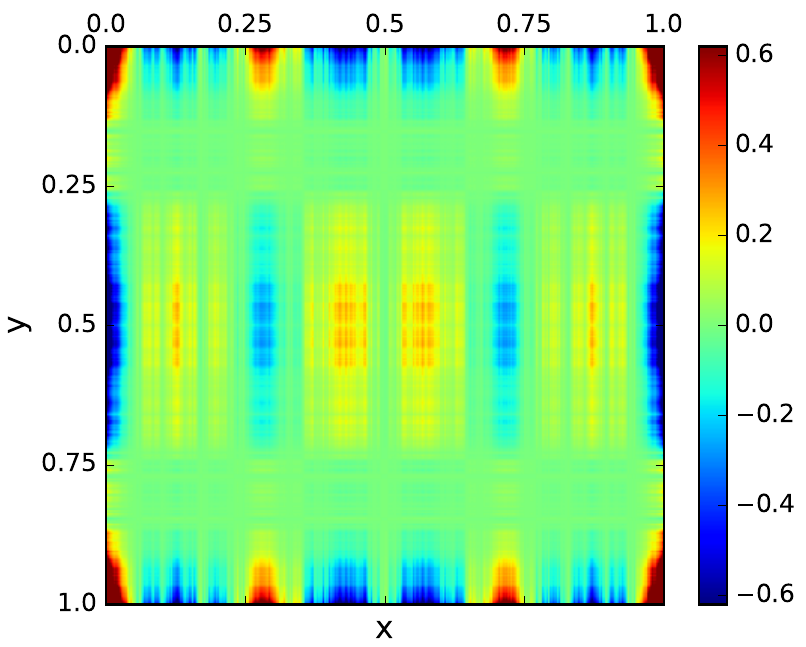}}
	\quad
	\subfigure[$u(t)\in H^{\frac12}\times H^{-\frac12}$ in one sample path]{\includegraphics[width=7.0cm,height=5.3cm]{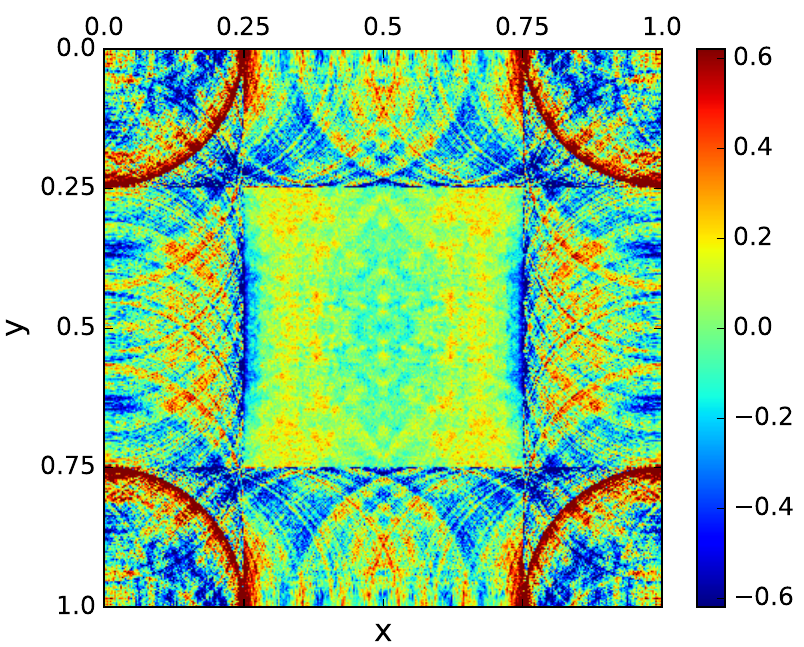}}
	\vfill\bigskip
	\centering
	\subfigure[$u_0\in H^{4}\times H^{3}$ in one sample path]{\includegraphics[width=7.0cm,height=5.3cm]{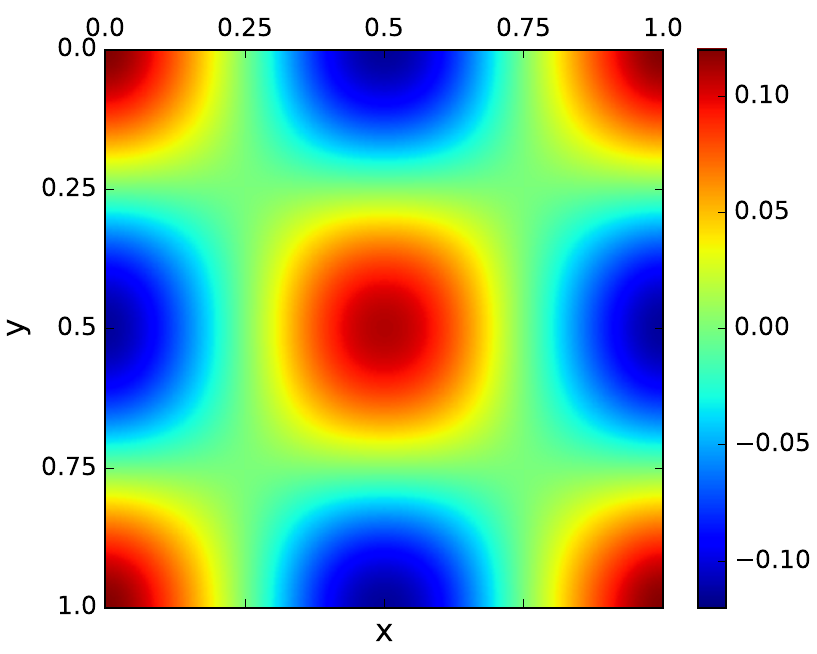}}
	\quad
	\subfigure[$u(t)\in H^{4}\times H^{3}$ in one sample path]{\includegraphics[width=7.0cm,height=5.3cm]{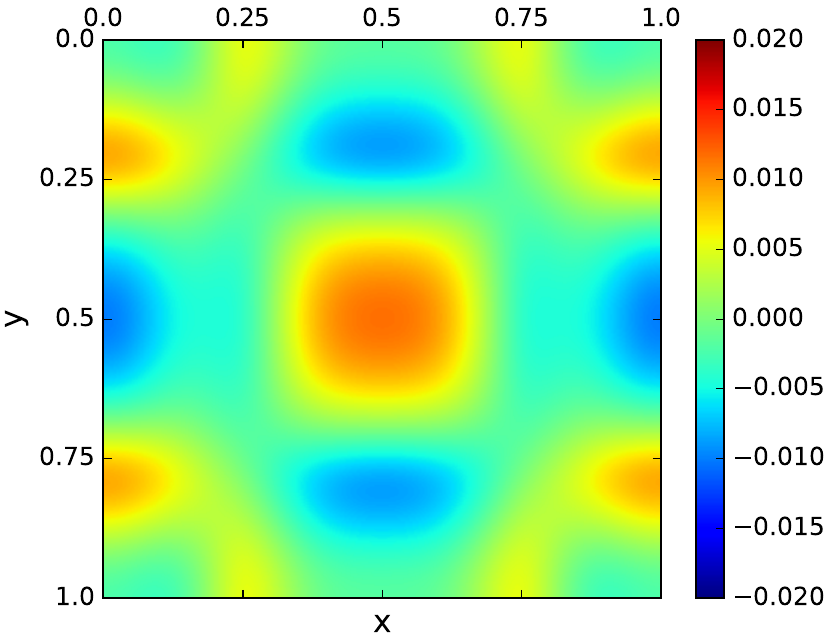}}
	\caption{Numerical solutions in two dimensions (Example \ref{Example4}).}
	\label{fig:4-5-0}
\end{figure}

\begin{figure}[!htbp]
\centering
\subfigure[$L^{2}(\Omega)\times H^{-1}(\Omega)$ error vs $\tau$ for $\gamma=\frac{1}{2}$]{\includegraphics[width=7.0cm,height=6.0cm]{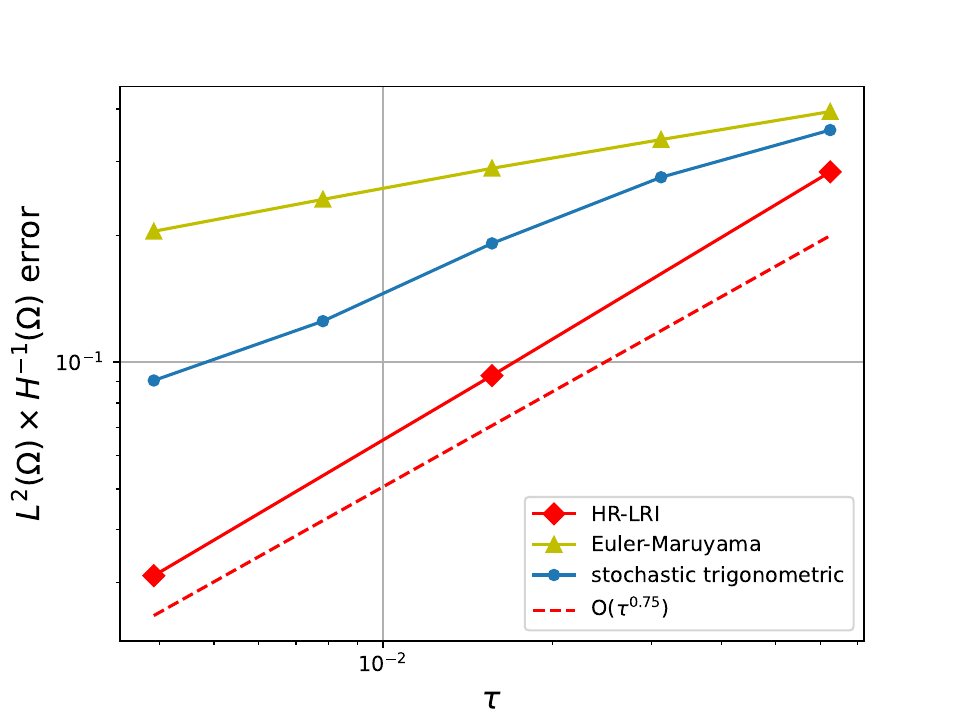}}
\quad
%\subfigure[$L^{2}(\Omega)\times H^{-1}(\Omega)$ error vs CPU time for $\gamma=\frac{1}{2}$]{\includegraphics[width=7.0cm,height=6.0cm]{fig/cpu_L2_error_H05_2d.pdf}}
%\vfill
%\centering
\subfigure[$L^{2}(\Omega)\times H^{-1}(\Omega)$ error vs $\tau$ for $\gamma=4$]{\includegraphics[width=7.0cm,height=6.0cm]{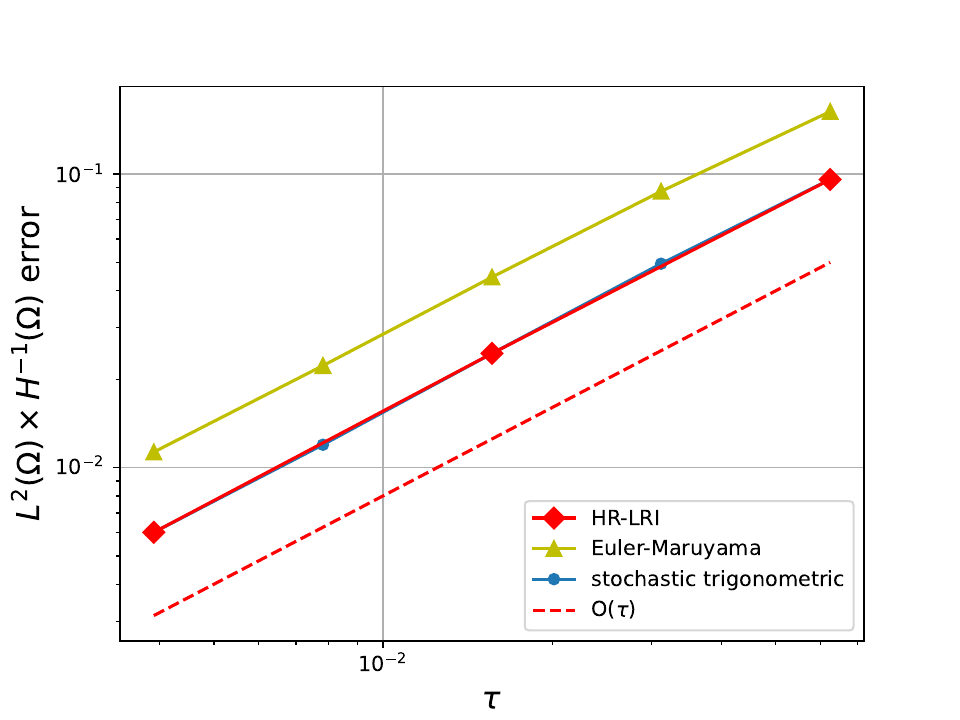}}
%\quad
%\subfigure[$L^{2}(\Omega)\times H^{-1}(\Omega)$ error vs CPU time for $\gamma=4$]{\includegraphics[width=7.0cm,height=6.0cm]{fig/cpu_L2_error_H4_2d.pdf}}
\caption{Errors of the numerical solutions by several methods (Example \ref{Example4}).}
\label{fig:4-5}
\end{figure}

%The numerical results confirm the theoretical results from Theorem~\ref{thm:0_norm_error_rough_data} for the two-dimensional case and demonstrate the efficiency of the proposed method in computing rough solutions of the stochastic wave equation. 

\end{example}

\pagebreak
\section{Conclusion}

By leveraging the structure of the stochastic nonlinear wave equation and utilizing harmonic analysis tools such as low- and high-frequency decomposition techniques, we have developed an efficient numerical algorithm for computing rough solutions of the stochastic nonlinear wave equation, while significantly relaxing the regularity requirements for the initial data. Specifically, we have proved that the proposed method achieves an error bound of $O(\tau^{2\gamma-})$ in one and two dimensions for initial values $(u^{0}, v^{0}) \in H^{\gamma} \times H^{\gamma-1}$ and $\gamma \in (0, \frac{1}{2}]$, where $\tau$ represents the time discretization step size. Notably, this convergence rate is twice as fast as that of existing methods for the stochastic nonlinear wave equation under the same regularity conditions. The effectiveness of our approach is further demonstrated by the numerical results presented in Examples \ref{Example1}--\ref{Example4}, which highlight the practical significance of the method.

We believe that the techniques developed in this paper can be extended to design higher-order convergent numerical methods for the stochastic wave equation (with relaxed regularity conditions on the initial data), as well as for other stochastic partial differential equations. These extensions will be explored in future studies.

\section*{Acknowledgement}
This work was supported in part by the CAS AMSS-PolyU Joint Laboratory of Applied Mathematics, Research Grants Council of Hong Kong (Project No. PolyU/RFS2324-5S03 and PolyU/GRF15306123) and internal grants of Hong Kong Polytechnic University (Project ID: P0045404).

\bibliographystyle{model1-num-names}

\end{document}